\documentclass[12pt]{amsart}

\usepackage{graphicx, epsfig}
\usepackage{amsfonts,amscd, amssymb,  epsf, epsfig, graphicx}
\usepackage{amssymb,amsmath,amsfonts,array,delarray,amsthm,graphics,graphicx}

\def\bcases{\begin{cases}}

\def\ecases{\end{cases}}

\newcommand{\bea}{\begin{eqnarray*}}
\newcommand{\eea}{\end{eqnarray*}}

\newtheorem{thm}{Theorem}
\newtheorem{cor}{Corollary}
\newtheorem{lemma}{Lemma}
\newtheorem{prop}{Proposition}

\theoremstyle{definition}
\newtheorem{defn}{Definition}

\newtheorem{example}{Example}
\theoremstyle{remark}
\newtheorem{remark}{Remark}

\newcommand{\be}{\begin{equation}}
\newcommand{\ee}{\end{equation}}

\title{Examples of minimal laminations and associated currents}

      \author{John Erik Forn\ae ss, Nessim Sibony, Erlend Forn\ae ss Wold}
      
\begin{document}
\maketitle

\begin{abstract}
In this paper, we construct various examples
of holomorphic laminations,with leaves of dimension 1, and we also study some of their dynamical properties.  In particular we study existence and uniqueness of positive closed  currents. We construct minimal laminations with infinitely many mutually singular closed currents and no non-closed harmonic current. We also consider embeddings to projective space. 
\end{abstract}
 
\section{Introduction}





There are several surveys on laminations by Riemann surfaces.  E. Ghys in \cite{Gh1999}
gives a nice introduction to the theory. The recent survey \cite{FS2008} by two of the authors focuses on holomorphic foliations with singularities, in particular ergodic properties of holomorphic foliations in $\mathbb P^2.$
There are few examples in the literature of foliations by Riemann surfaces in higher 
dimensional projective spaces or with leaves of dimension $>1$, see however 
Sullivan \cite{S1976}, Ghys \cite{Gh1999}, Loray-Rebelo \cite{LR2003},  
Candel-Conlon \cite{CC1999}, L. Garnett
\cite{G1983}, Deroin \cite{D2003}, and \cite{FS2008}.

In this paper we construct some examples and we explain their dynamics. We consider the problem of embedding them into $\mathbb P^k.$
We also study the properties of directed positive 
$\partial \overline{\partial}$-closed currents associated to these laminations.

In section 2 we construct laminations as projective limits of sequences 
of compact complex 
manifolds.  These laminations admit unique directed closed currents of mass one.  We then
show that these projective limits
embed to $\mathbb{CP}^k$.  The embedded laminations admit unique 
closed directed currents of mass one and every 
$\partial \overline{\partial}$-closed directed current is closed.  In the Riemann surface
case we construct projective limits in a way that allows us to control the topology of the leaves.  
(In \cite{Gh1999} p.52, Ghys has constructed a laminated set in $\mathbb P^3$ using a suspension.)

In section 3 we explore the properties of some foliations obtained by suspension using some 
remarkable diffeomorphisms of the real $2-$torus constructed by Furstenberg \cite{Fu1961}. In particular we exhibit a minimal lamination by Riemann surfaces with uncountably many extremal positive directed closed currents and for which every directed $\partial \overline{\partial}-$closed current is closed.

In section  4 we give an abstract criteria for the existence of positive 
$\partial\overline\partial$-closed currents.    
We construct laminations with infinitely many extremal positive 
$\partial \overline{\partial}$-closed currents, and, opposed to the Riemann 
surface case, we construct a two dimensional lamination in $\mathbb P^2\times\mathbb P^2$ 
with no positive 
$\partial\overline\partial$-closed current.  

In section 5 we discuss an interesting functional and use it to give examples 
of laminations with no non-closed positive $\partial\overline\partial$-closed
currents. 


\section{Construction of laminations as projective limits}

\subsection{Preliminaries}

\begin{defn}
Let $Y$ be a Hausdorff topological  space. 
Then $(Y,\mathcal L)$ is a lamination by complex manifolds 
of dimension $k$ if
${\mathcal L}$ is an atlas with charts
$$
\phi_{U}:U \rightarrow B \times T_U
$$
\noindent where $B$ is the unit ball in $\mathbb C^{k}$, 
$T_U$ is a topological space and $\phi$
is a homeomorphism.  The space $\phi_{U}^{-1}(\{z\}\times T_{U})$
is  called a \emph{transversal}.  It is sometimes identified with $T_{U}$. 
The change of coordinates should be of the form
$$
(z,t) \overset{\phi_{U,V}}{\rightarrow} (z',t'),\; {\mbox{with}}\; z'=h(z,t),t'=t'(t)
$$
\noindent where $h$ is continuous and holomorphic with respect to $z.$
A \emph{plaque} is a set of the form $\phi_{U}^{-1}(B\times\{t\})$. 
A \emph{leaf} is a minimal connected set $L$ such that if a plaque
intersects $L$ then it is contained in $L$.  The open set U is called a flow box.  A
lamination is called minimal if all leaves are dense.
\end{defn}

On $B\times T_{U}$ we consider the space of partially smooth $l$-forms,
$l\leq 2k$, \emph{i.e.}, forms that are smooth on plaques, that depend continuously 
on the parameter $t$, and whose derivatives depend continuously on $t$.  
The topology is the $\mathcal C^\infty$-topology
on plaques.  $\mathcal A^l(\mathcal L)$ will denote the space of $l$-forms
on the lamination, \emph{i.e.}, if $\varphi$ is in $\mathcal A^l(\mathcal L)$
then $(\phi_{U})_{*}\varphi$ is an $l$-form on $B\times T_{U}$.  So for each t
we have an $l$-form on B, depending continuously on t. 
This space is a Fr\'{e}chet space and it is easy to define the Poincar\'{e}
operator $d:\mathcal A^l(\mathcal L)\rightarrow\mathcal A^{l+1}(\mathcal L)$, 
the operator being exterior differentiation on leaves.  It is also easy to define 
the space $\mathcal A^{p,q}(\mathcal L)$ of $(p,q)$-forms on $\mathcal L$
and the operators $\partial$ and $\overline\partial$ acting on leaves. 
Currents of dimension $l$ are defined as continuous linear functionals on $\mathcal A^l(\mathcal L)$. 
We note that this is different from the directed currents one usually studies on 
embedded laminations in complex manifolds.  In that case the test forms are the 
ones restricted from the ambient space, \emph{i.e.}, there are fewer test forms and 
more currents (a priori). 
When the current defines a linear form on $\mathcal A^{(p,p)}(\mathcal L)$ we say that it is of 
bidimension (p,p).

\

The following local representation corresponds to Theorem I.12 in \cite{S1976}.

\begin{prop}\label{local}
Let $(Y,\mathcal L)$ be a k-dimensional lamination.   If $T$ is a 
closed current of order zero and of bidimension (k,k) on $(Y,\mathcal L)$
then $T$ is locally of the form 
$$
T=\int_{T_{U}} [B]\mathrm{d}\mu,
$$
where $\mu$ is a measure on the transversal $T_{U}$ and [B] denotes the current of integration on B.  If $\partial\overline\partial T=0$ 
then $T$ is locally of the form
$$
T=\int_{T_{U}} h_{t}\cdot[B]\mathrm{d}\mu,
$$
where $h_{t}$ is a positive pluriharmonic function.   
\end{prop}
\begin{proof}
The proof is the same as in \cite{S1976}.  Locally $T$ is represented by a current of maximal dimension. 
This current disintegrates and by studying the extremal elements using cutoff 
functions one concludes that they are closed (resp. pluriharmonic) on plaques.    
\end{proof}
 
 \begin{remark}
 Forn\ae ss-Wang-Wold has a counterexample to the corresponding result 
 for positive closed (resp. harmonic) currents directed by an embedded lamination.  
 The result holds however in complex dimension two \cite{FWW2008}. 
\end{remark}

\subsection{Projective Limits}

Consider a sequence $\{X_n\}_{n=1}^\infty$ of compact complex manifolds of dimension 
$d\geq 1$.
Assume there is a holomorphic covering map $f_n: X_{n+1}\rightarrow X_n$ of degree $d_n \geq 2.$
Consider the projective limit of the pairs $(X_n,f_n)$
$$
X_\infty = \lim_{\leftarrow} X_n.
$$

Recall that $x=\{x_n\}$ is an element of $X_\infty$ if the coordinates satisfy the relation $x_n=f_n(x_{n+1}).$
The topology of $X_\infty$ is the weakest topology such that the projections $\pi_n:X_\infty \rightarrow
X_n$ are continuous. Define $h_n= f_1 \circ \cdots \circ f_{n-1}$. 
Then $X_\infty$ is a compact space and can be given a natural structure of a holomorphic lamination such that the maps $\pi_n$ are holomorphic on leaves.
More precisely, let $p\in U\subset X_1$ be a small open set in $X_1.$ 
Define the transversal $T$ by $\pi_1^{-1}(p).$ 
The local chart is given by the sequences of all local sections of the $h_n$ through each preimage of $p.$
Each leaf is a covering manifold of each of the manifolds $X_{n}$.  
A basis for the topology consists of sets of the following type.  Let $U^m$ be an 
open subset of $X_{m}$ for some $m\geq 1$.  Then $U^m_{\infty}$ is defined as 
the set of points $\{x_{j}\}_{j=1}^\infty$ such that 
$x_{m}\in U^m$.  

We will consider that $X_1$ has a Hermitian form $g_1$ and that $X_n$ is endowed with the 
form $(h_n)^*g_1.$ So we have a local isometry between $X_n$ and $X_1$. 
We denote by $d_n$ the metric on $X_n$, and $d_\infty ({x_n},{x'_n}):=\underset{n}
{\sup} \ d_n(x_n,x'_n).$  Observe that $d_{\infty}$ takes values in $[0,\infty]$.

\begin{prop}
Let $X_{\infty}$ be a projective limit.  Then each transversal is a Cantor set
and each leaf is dense.  Two points ${x_n}$ and ${x'_n}$ are in the same leaf if and only if 
the sequence $d_n(x_n,x'_n)$ is bounded.
\end{prop}
\begin{proof}
Let $U_{1}$ be simply connected open subset of $X_{1}$ and let $U_{m}\subset X_{m}$
be connected such that $f_{1}\circ\cdot\cdot\cdot\circ f_{m-1}(U_{m})=U_{1}$. 
By using the basis for the topology we see that 
$\pi_{1}^{-1}(U_{1})\setminus U^m_{\infty}$ is open, and so the topological 
space of plaques over $U_{1}$ is totally disconnected.  In particular each transversal 
is totally disconnected.  That each transversal is perfect will follow from the proof 
that each leaf is dense. 

Let $x_{1}\in X_{1}$ be a point.  We first show that any leaf contains 
a point that projects to $x_{1}$.  Let $\{x_{n}'\}\in X_{\infty}$.
Choose an arc $\gamma$ between $x_{1}'$ and $x_{1}$.  For each $n\in\mathbb N$
the lifting of $\gamma$ with initial point $x_{n}'$ determines a lifting of $x_{1}$
to a point $x_{n}$, \emph{i.e.}, we obtain a point $\{x_{n}\}$ 
in the leaf through $\{x_{n}'\}$.
Clearly $\gamma$ lifts to a path between $\{x_{n}\}$  and $\{x_{n}'\}$.  

Let $\{x_{n}\}$ and $\{x_{n}'\}$ be 
two points with $x_{1}=x_{1}'$, and let $L$ and $L'$ denote their leaves 
respectively.  We will show that there are points $\{x_{n}''\}$
in $L'$ with $x_{1}''=x_{1}$ arbitrarily close to $\{x_{n}\}$. 

Unless the two points coincide there exists an integer $k$ such that $x_{k}\neq x_{k}'$.
 Choose an arc $\gamma$ between $x_{k}'$
and $x_{k}$.  Let $x_{n}''=x_{n}$ for $n=1,...,k-1$.  For $n\geq k$ let 
$x_{n}''$ be the lifting to $X_{n}$ determined by the lifting of $\gamma$
with initial point $x_{n}'$.  Then $\{x_{n}''\}$  
is in the same leaf as $\{x_{n}'\}$, and the bigger $k$ is, 
the closer the two points are together.       

Now assume that $\{x_{n}\}$ and $\{x_{n}'\}$ are in the same leaf.  Then there 
is a path $\tilde\gamma$ in the leaf connecting them.  
Then $\gamma_{n}:=\pi_{n}(\tilde\gamma)$
is a path in $X_{n}$ connecting $x_{n}$ and $x_{n}'$, and $d_{n}(x_{n},x_{n}')$
is less than the length of $\gamma_{n}$.  The length of $\gamma_{n}$ is the same for all $n$. 

If on the other hand $d_{n}(x_{n},x_{n}')<R\in\mathbb R^+$ for all $n$ there 
is a path $\gamma_{n}$ of length less than R connecting them in $X_{n}$ for all $n$. 
Consider the collection of projections 
$f_{1}\circ\cdot\cdot\cdot\circ f_{n-1}(\gamma_{n})$.  
These are representatives of elements 
of $\pi_{1}(X_{1},x_{1})$ all of length less than $R$ and so there are only a finite number 
of classes.  So we may choose a representative $\gamma_{1}$ whose class occur infinitely many times. 
But then $x_{n}'$ is the point in $X_{n}$ determined by the lift of $\gamma_{1}$ with 
initial point $x_{n}$, so the curve lifts to $X_{\infty}$ and connects the two points.  
\end{proof}

\begin{thm}\label{current}
Let $X_{\infty}$ be a projective limit of dimension k.  
Then $X_{\infty}$ supports a positive closed (k,k)-current, 
and any $\partial\overline\partial$-closed (k,k)-current $T$ of order $0$ on $X_{\infty}$  is closed.  It is uniquely determined by the values $\langle (\pi_1)_*T,\omega\rangle$ where $\omega$ is a fixed volume form on $X_1.$ 
The support of $T$ is $X_{\infty}$.  
\end{thm}

The proof of the theorem depends on an approximation result 
that we will prove first.  See also the thesis of Deroin \cite{D2003}.

\begin{prop}\label{approx1}
Let $\Omega\subset\mathcal{A}^{(k,k)}(X_{\infty})$  
denote the space of forms $\underset{n\in\mathbb N}{\bigcup}\pi_{n}^*
(\mathcal{A}^{(k,k)}(X_{n}))$.  Then $\Omega$ is dense in $\mathcal{A}^{(k,k)}(X_{\infty})$.
\end{prop}
\begin{proof}
Let $\alpha\in\mathcal{A}^{(k,k)}(X_{\infty})$.  
Since $X_{1}$ is compact we may cover it by a finite number of 
coordinate charts $\mathcal B=\{B_{j}\}_{j=1}^s$ biholomorphic to balls in 
$\mathbb C^k$.
Let $\{\beta_{j}\}$ be a partition of unity with respect to $\mathcal B$.  Then 
$$
\alpha=\sum_{j=1}^\ell \alpha_{j}:=\sum_{j=1}^\ell(\beta_{j}\circ\pi_{1})\cdot\alpha.
$$ 
We will show that we can approximate each $\alpha_{j}$ arbitrarily well.  \

Choose coordinates $\phi_{j}:\pi_{1}^{-1}(B_{j})\rightarrow B\times T^j_{B}$. 
In these coordinates we have that $\alpha_{j}$ is given by 
$\alpha_{j}(z,t)=f_{j}(z,t)\mathrm{d}z\wedge\mathrm{d}\overline z=
f_{j}(z,t)dz_{1}\wedge\cdot\cdot\cdot\wedge dz_{k}\wedge d\overline{z_{1}}
\wedge\cdot\cdot\cdot\wedge d\overline{z_{k}}$.
So it is enough to show that we may approximate the function $g_{j}=f_{j}\circ\phi$
on $\pi_{1}^{-1}(B_{j})$.  \

Let $\epsilon>0$.  A compactness argument shows that there exists an $m\in\mathbb N$
such that for all $\{x_{n}\}$ and $\{x'_{n}\}$
with $x_{n}=x'_{n}$ for $n=1,...,m$ and $x_{1}\in B_{j}$, we have that 
$|g_{j}(\{x_{n}\})-g_{j}(\{x'_{n}\})|<\epsilon$ 
and similarly for any given number of derivatives.
Let $V_{1},...,V_{\Pi_{j=1}^{m-1}d_{j}}$ denote the disjoint preimages of $B_{j}$ in $X_{m}$.
Define a function $\tilde g_{j}$ on $V_{s}$ as follows: choose any leaf
$L_{s}$ in $\pi_{m}^{-1}(V_{s})$ and define $\tilde g_{j}=g_{j}\circ\pi_{m}^{-1}$.
Then $|\tilde g_{j}\circ\pi_{m}(x)-g_{j}(x)|<\epsilon$ for all $x\in X_{\infty}$.
\end{proof}

\emph{Proof of Theorem \ref{current}:}
For each $n\geq 1$ let $T_{n}$ denote the current of integration on $X_{n}$.
With respect to some volume form $\omega_{1}$ on $X_{1}$ the current 
$T_{1}$ has total mass one.  Next let $\tilde T_{1}:=T_{1}$ and for each $n\geq 2$
let $\tilde T_{n}:=\frac{1}{\Pi_{j=1}^{n-1}d_{j}}T_{n}$.  Then $T_{n}$
has total mass one with respect to the volume form $\omega_{n}$ defined 
inductively by $\omega_{n}:=f_{n-1}^*\omega_{n-1}$ and for any (k,k)-form 
$\alpha$ on $X_{n-1}$ we have that $\tilde T_{n}(f_{n}^*\alpha)=\tilde 
T_{n-1}(\alpha)$.  We define the current $T$ on $\Omega$ as follows.  For 
an element $\alpha=\pi_{n}^*\alpha_{n}$ let $T(\alpha)=\tilde T_{n}(\alpha_{n})$.
This current is closed and of order zero on $\Omega$ and extends to a closed 
current of order zero 
on $\overline\Omega$ which according to Proposition \ref{approx1} is 
$\mathcal A^{(k,k)}(X_{\infty})$.  \

Assume next that $T'$ is a normalized (k,k)-current of order zero on $X_{\infty}$ such 
that $\partial\overline\partial T=0$.  Then $T'_{n}:=(\pi_{n})_{*}T$
satisfies $\partial\overline\partial T'_{n}=0$ for each $n$ and so 
$\mathrm{d}T'_{n}=0$ for each $n$.  It follows that $T_{n}'=T_{n}$ and so $T'=T$.
Since all leaves are dense it follows that $T$ has mass everywhere.   
$\hfill\square$

\medskip

\begin{prop}\label{approx2}
Let $X_{\infty}$ be a projective lamination, and let $U$ be 
a chart with $\phi_{U}:U\rightarrow B\times T_{U}$. 
Let $\varphi:X_{\infty}|_{U}\rightarrow Y$ be an embedding 
to a complex manifold $Y$.  Then any compactly supported (k,k)-form
$\omega$ on $U$ is the uniform limit of forms $\varphi^*(\omega_{j})$, 
where the $\omega_{j}$'s are - in the usual sense - smooth (k,k)-forms 
on $Y$.  
\end{prop}
\begin{proof}
The proof is essentially the same as that of Proposition \ref{approx1} since locally 
the collection of embedded plaques is totally disconnected. 
\end{proof}
\begin{cor}
Let $X_{\infty}$ be a projective lamination of dimension $k$ 
and let $\phi:X_{\infty}\rightarrow Y$
be an embedding in a complex manifold $Y$.  There is 
one to one correspondence between (k,k)-currents of order 
zero on $X_{\infty}$ and (k,k)-currents of order zero on 
$Y$ which are weakly directed by $\phi(X_{\infty})$.    
\end{cor}

Recall that a current $T$ is weakly directed by a lamination $\mathcal L$ if and only 
if $T \wedge \alpha=0$ for every continuous one form $\alpha$, 
vanishing on the plaques of $\mathcal L$, 
\emph{i.e.} $\alpha \wedge [L]=0$ for every plaque $L.$

\subsubsection{Embeddings in $\mathbb{CP}^N$}

We will now consider embeddings of projective limits in projective space.  
By an embedding 
$\Phi:X_{\infty}\hookrightarrow\mathbb{CP}^N$ we will mean 
a topological embedding, \emph{i.e.}, $\Phi$ is a homeomorphism onto its image, which 
is holomorphic and of maximal rank along leaves.  

\begin{thm}
Let $X_{1}$ be a complex projective manifold of dimension $n$ 
and let $X_{j+1}\overset{f_{j}}{\rightarrow}{X_{j}}$ 
be a projective limit over $X_{1}$.
Then $X_{\infty}$ admits an embedding into $\mathbb{CP}^{2n+1}$.
\end{thm}
\begin{proof} 

To prove this result we will use $L^2$-techniques for embedding projective manifolds, 
or, equivalently, manifolds admitting positive line bundles.  We will refer to Demailly \cite{De}, 
and we start by recalling some notions and results (from \cite{De}) and we describe a technique for producing embeddings. 

Let $X$ be a compact complex manifold and let $L\rightarrow X$ be a complex line bundle.  
A singular metric on $L$ is a metric which is given in any trivialization $\Theta:L_{\Omega}\rightarrow\Omega\times\mathbb C$ by
$$
\|\xi\|=|\Theta(\xi)|\mathrm{e}^{-\varphi(x)}, x\in\Omega, \xi\in L_{x},
$$
where $\varphi\in L^1_{\mathrm{loc}}(\Omega)$ is an arbitrary function.  The curvature form 
of $L$ is given by the closed (1,1)-current $c(L):=\frac{i}{\pi}\partial\overline\partial\varphi$.  
We will consider line bundles with singular metrics and with positive curvature.  The Lelong number of $\varphi$
at $x$ is given by 
$$
\nu(\varphi,x)=\underset{z\rightarrow x}{\mathrm{{lim} \ {inf}}}\frac{\varphi(z)}{\mathrm{log}|z-x|}.
$$
The first important thing to recall is that if $\varphi$ is a plurisubharmonic function on $X$ then \emph{$\mathrm{e}^{-2\varphi}$ is non integrable in 
a neighborhood of $x$ when $\nu(\varphi,x)\geq n$}.  

\begin{lemma}
Let $X$ be a compact complex manifold with a K\"{a}hler metric $\omega$.  
Let $L\rightarrow X$ be a line bundle with a strictly positive metric $\mathrm{e}^{-\varphi}$.
Then there exists a number $k\in\mathbb N$
such that for any 
two points $x_{1},x_{2}\in X$ there exists a singular metric $\tilde\varphi$ on $L^{\otimes k}$
such that $c(L^{\otimes k})\geq\omega$ and such that $\nu(\tilde\varphi,x_{j})=n+1$ for $j=1,2.$
\end{lemma}
\begin{proof}
Let $\{U_{j}\}_{j=1}^m$ be a holomorphic cover of $X$ with $f_{j}:\overline{U_{j}}\rightarrow\overline{2\mathbb B^n}$. 
We may assume that the sets $\{f_{j}^{-1}(\frac{1}{4}\overline{\mathbb B^n})\}$ cover $X$.   Let $\chi\in\mathcal C^\infty_{0}(2\mathbb B^n)$
be a smooth cut-off function with $\chi\equiv 1$ on $\frac{3}{2}\mathbb B^n$.  
Let $\psi\in\mathcal C^\infty_{0}(\frac{3}{4}\mathbb B^n)$
be a smooth cut-off function with $\psi\equiv 1$ in on $\frac{1}{2}\mathbb B^n$.  

For any point $x\in X$ choose an open set, say $U_{j}$, such that $\|f_{j}(x)\|\leq\frac{1}{4}$.
Write $a=f_{j}(x)$.  We define in local coordinates $\sigma(j,a)(z)=\chi(z)\cdot\mathrm{log}\|z-a\|^{n+1}$.
Note that $i\partial\overline\partial\sigma(j,a)$ is a positive current on $\mathbb B^n(a,1)$.   Consider 
the current $\tilde\sigma(j,a):=(1-\psi)\cdot i\partial\overline\partial\sigma(j,a)$.  This is a smooth (1,1)-form
and so by the strict positivity of $i\partial\overline\partial\varphi$ there exists an $s\in\mathbb N$
such that $(1-\psi)i\partial\overline\partial(s\cdot\varphi + \tilde\sigma(j,a))\geq(1-\psi)\omega$
and so that $i\partial\overline\partial s\varphi\geq\omega$.  

Now let $\gamma$ be a positive test form on $U_{j}$.  Then 

\bea
\langle i\partial\overline\partial(s\varphi+
\sigma(j,a)),\gamma\rangle &=&
\langle i\partial\overline\partial(s\varphi+
\sigma(j,a)),\psi\gamma+(1-\psi)\gamma\rangle \\
& \geq & \langle i\partial\overline\partial(s\varphi),\psi\gamma\rangle + \langle i\partial\overline\partial(s\varphi+
\sigma(j,a)),(1-\psi)\gamma\rangle\\
& \geq & \langle\psi\omega,\gamma\rangle + 
\langle(1-\psi)\omega,\gamma\rangle\\
& = & \omega(\gamma).
\eea
Finally note that, by compactness, the integer $s\in\mathbb N$ may be chosen independently of the point $a$, 
and by finiteness it can also we chosen independently of $j$.    
Fix such an $s$, define $k=2s$, and 
for each pair of points $x_{1}$ and $x_{2}$ let $a_{1}=f_{j}(x_{1})$ and $a_{2}=f_{k}(x_{2})$
for suitable maps $f_{j}$ and $f_{k}$, and use the metric 
$\mathrm{e}^{-(k\varphi+\sigma(j_{1},a_{1})+
\sigma(j_2,a_{2}))}$
on $L^{\otimes k}$. 
\end{proof}

We now draw the conclusion that will enable us to embed projective limits over projective manifolds into projective space. 
Let $X$ be a projective manifold.  We equip $X$ with the induced K\"{a}hler metric $\omega$ and a
strictly positive line bundle $L$.  By the above considerations we may assume
that 

\begin{itemize}
\item[a.] For any two points $x_{1}$ and $x_{2}$ in $X$ there exists a singular metric $\mathrm{e}^{-\varphi_{x_{1},x_{2}}}$
on $L$ such that $\nu(\varphi_{x_{1},x_{2}},x_{j})=n+1$ for $j=1,2,$ and $c(L)\geq\omega$.  
\end{itemize}
What is important to note for us is the fact that if 
$\tilde X\overset{\tilde f}{\longrightarrow} X$ is an unbranched covering map then   
\begin{itemize}
\item[b.] For any two points $\tilde x_{1}$ and $\tilde x_{2}$ in $\tilde X$ 
there exists a singular metric $\mathrm{e}^{-\varphi_{\tilde x_{1},\tilde x_{2}}}$
on $\tilde f^*(L)$ such that $\nu(\varphi_{\tilde x_{1},\tilde x_{2}},\tilde x_{j})=n+1$ for $j=1,2,$ and 
$c(\tilde f^*(L))\geq\tilde f^*(\omega)$.  
\end{itemize}
(simply pull the metric back from $X$).

For a given complex projective manifold $X_{1}$ we now fix such a "good" line bundle $L$ 
over it. 
Let $V$ denote the holomorphic vector bundle $V:=\bigwedge^{n,0}T^*X_{1}$ over $X_{1}$.
The embedding $\Phi:X_{\infty}\rightarrow\mathbb{CP}^{2n+1}$ will be constructed 
by an inductive procedure.  We will start by  
constructing an initial embedding $\varphi_{1}:X_{1}
\hookrightarrow\mathbb{CP}^{2n+1}$; this embedding will be 
given by sections $h_{0},...,h_{2n+1}$ in the bundle $V\otimes L$.  
For the inductive step we will assume that we are given an embedding $\varphi_{n}:$ 
$X_{n}\hookrightarrow\mathbb{CP}^{2n+1}$ given by sections 
in $f(n-1)^*V\otimes f(n-1)^*(L)$ and approximate the immersion $f_{n}\circ\varphi_{n}$
by sections in $f(n)^*V\otimes f(n)^*L$ (here $f(n)$ denotes the composition $f_{n}\circ\cdot\cdot\cdot\circ f_{1}$).
The proof of the existence of the initial embedding and the key part of the inductive step is furnished 
by the following theorem  and Lemma \ref{projection} below.    

\begin{thm}\label{embedding}
Let $\tilde X\overset{\tilde f}{\longrightarrow} X$ be an unbranched covering of $X$. Then 
there exist holomorphic sections $h_{0},...,h_{N}$ of $\tilde f^*V\otimes f^*L$ such that 
$$
h:=[h_{0}:\cdot\cdot\cdot:h_{N}]:\tilde X\rightarrow\mathbb{CP}^N
$$
is an embedding. 
\end{thm}
\begin{proof}
Note that when $\tilde X=X$ and $\tilde f$ is the identity map this is essentially 
the content of Kodaira's embedding theorem.  
For any two points $\tilde x_{1}$ and $\tilde x_{2}$ we have to produce sections that separate the points, and 
for any one point $\tilde x_{1}$ we have to produce sections with non-vanishing differentials.  We show how to 
separate points, controlling differentials is similar.  Both are standard constructions.   

Let $u_{1}$ be a smooth (n,0)-form with coefficients in $\tilde f^*L$ such that $u_{1}$ is holomorphic and non-zero
near $\tilde x_{1}$ and such that $u_{1}$ is constantly zero near $\tilde x_{2}$.  Let $u_{2}$ be 
a smooth $(n,0)$-form with coefficients in $\tilde f^*L$ such that $u_{2}$ is holomorphic and non-zero near both points
$\tilde x_{1}$ and $\tilde x_{2}$.  Let $v_{j}=\overline\partial u_{j}$ for $j=1,2$.
Then each $v_{j}\in L^2(\varphi_{\tilde x_{1},\tilde x_{2}})$.  According to Theorem 3.1 in 
\cite{De} there exist smooth (n,0)-forms $w_{j}\in L^2(\varphi_{\tilde x_{1},\tilde x_{2}})$
such that $\overline\partial w_{j}=v_{j}$, and both $w_{j}$ need to vanish at the points $\tilde x_{j}$.
Then $h_{j}:=u_{j}-w_{j}$ are holomorphic sections and we see that the function $\frac{h_{1}}{h_{2}}$
separates the two points.  
\end{proof}

We proceed to describe the inductive procedure.  As noted above the initial embedding 
of $X_{1}$ is constructed by applying Theorem \ref{embedding} with $\tilde X=X$
and then Lemma \ref{projection} below.  

The plan is to inductively construct embeddings 
$\varphi_{n}:X_{n}\rightarrow\mathbb{CP}^{2n+1}$
approximating the immersions $\varphi_{n-1}\circ f_{n-1}$ and then define 
\begin{itemize}
\item[(a)] $\Phi(\{x_{n}\}_{n=1}):=\underset{n\rightarrow\infty}{\mathrm{lim}}
\varphi_{n}(x_{n})$. 
\end{itemize}

Let $d_{1}$ and $d_{2}$ be Riemannian metrics on $X_{1}$ and $\mathbb{CP}^{2n+1}$
respectively and for any two maps $g_{i}:X_{1}\rightarrow\mathbb{CP}^{2n+1}$ and a closed 
set $K\subset X_{1}$ let $\|g_{1}-g_{2}\|_{K}$ denote the maximum distance
between images $g_{1}(x)$ and $g_{2}(x)$ for $x\in K$ with respect the distance $d_{2}$. 
Let $U_{j}'\subset\subset U_{j}$ be open balls
for $j=1,...,m_0$ such that $\{U_{j}'\}_{j=1}^{m_0}$ is a cover of $X_{1}$.  
Choose $\delta>0$ such that for any two points $x_{1},x_{2}\in X_{1}$ either 
$d_{1}(x_{1},x_{2})\geq\delta$ or $x_{1},x_{2}\in U_{j}'$ for at least one $j$.
Choose 
$\epsilon>0$ such that if $g:\overline{U_{j}}\rightarrow\mathbb{CP}^{2n+1}$
is a holomorphic map with $\|g-\varphi_{1}\|_{\overline{U_{j}}}<\epsilon$
then $g|_{\overline{U_{j}'}}$ is an embedding and such that for any pair of points
$x_{1},x_{2}\in X_{1}$ we have that $d_{1}(x_{1},x_{2})\geq\delta$
implies that $d_{2}(\varphi_{1}(x_{1}),\varphi_{1}(x_{2}))\geq\epsilon$.   

\medskip

We now describe the inductive procedure to construct an embedding $\varphi_{n+1}$
and a constant $\epsilon_{n+1}$ given embeddings $\varphi_{1},...,\varphi_{n}$
and constants $\epsilon_{1},...,\epsilon_{n}$.   Let $\epsilon_{1}=\epsilon$.

We claim first that we can construct $\varphi_{n+1}:X_{n+1}\rightarrow\mathbb{CP}^{2n+1}$
such that 
\begin{itemize}
\item[(b)] $\|\varphi_{n+1}-\varphi_{n}\circ f_{n}\|_{X_{n+1}}<
(\frac{1}{2})^{n+1}\cdot\frac{\epsilon_{n}}{3}$.
\end{itemize}
The immersion $\varphi_{n}\circ f_{n}$ is given by sections $[h_{0};\cdot\cdot\cdot;h_{2n+1}]$
in the bundle $f(n)^*V\otimes f(n)^*L$.  According to Theorem \ref{embedding} there exist 
sections $\tilde h_{1},...,\tilde h_{N}$ such that 
$$
[h_{0}:\cdot\cdot\cdot h_{2n+1}:\tilde h_{1}:\cdot\cdot\cdot:\tilde h_{N}]
$$
is an embedding.  The claim then follows from Lemma \ref{projection} below.

Next we define $\epsilon_{n+1}$.  For each $j$ let $s^j_{k}:U_{j}\rightarrow X_{n+1}$ denote the 
$\Pi_{j=0}^{n}d_{j}$ sections over 
$f(n):=f_{1}\circ\cdot\cdot\cdot\circ f_{n}:X_{n+1}\rightarrow X_{1}.$  Define 
\begin{itemize}
\item[(c)] $
\tilde\epsilon_{n+1}:=\underset{1\leq j\leq m_0, k_{1}\neq k_{2}}
{\mathrm{min}}\{\mathrm{dist}(\varphi_{n+1}\circ s^j_{k_{1}}(\overline{U_{j}'})
,\varphi_{n+1}\circ s^j_{k_{2}}(\overline{U_{j}'}))\}
$, and 
\item[(d)] $\epsilon_{n+1}:=\mathrm{min}\{\tilde\epsilon_{n+1},\epsilon_{n}\}$.
\end{itemize}
Since each $U_{j}$ is simply connected we have that $\epsilon_{n+1}>0$.  \

Now define $\Phi$ as in (a).  The map is well defined because of (b).  Let $\tilde U'_{j}$
be any lift of $U'_{j}$ to $X_{\infty}$.  By the choice of $\epsilon$ and (b) we see that 
$\Phi|_{\tilde U_{j}'}$ is an embedding.  To see that $\Phi$ is injective let $\{x_{n}\}$
and $\{x_{n}'\}$ be different points in $X_{\infty}$.  
If $d_{1}(x_{1},x_{1}')\geq\delta$
it follows by the choice of $\epsilon$ and (b) that $d_{2}(\Phi(\{x_{n}\}),\Phi(\{x_{n}'\}))
\geq\frac{\epsilon}{3}$.
By the choice of $\delta$, the remaining case is the case when 
$x_{1}$ and $x_{1}'$ are contained in the same $U_{j}'$.  Let $\{x_{1}''\}$ be 
the point where $x_{1}''=x_{1}$ and each $x_{n}''$ is determined by the lifting 
of $x_{n}'$ to $X_{n}$.  Let $m+1$ be the smallest integer such that $x_{m}''\neq x_{m}$.
By definitions (c) and (d) 
we have that $d_{2}(\varphi_{m+1}(x_{m+1}'),\varphi_{m+1}(x_{m+1}))
\geq\epsilon_{m+1}$.  
By (b) we get that $d_{2}(\Phi(\{x_{n}\}),\Phi(\{x_{n}'\}))\geq
\frac{\epsilon_{n+1}}{3}$.  On the other hand, if $\{x_{n}'''\}$ is 
a point with $x_{n}'''=x_{j}$ for $j=1,...,m+1$ we get from (b) that 
$d_{2}(\Phi(\{x_{n}'''\}),\Phi(\{x_{n}\}))<\epsilon_{m+1}$, and this concludes 
the argument that $\Phi$ is an embedding. 
\end{proof}

We now recall the projection result used in the above proof.  
We define projections
$\pi_{k}:\mathbb P^N \rightarrow \mathbb P^k$ for $N >k$ by
$$
\pi_{k}([x_0:\cdots :x_N]):=[x_0: \cdots :x_k].
$$ 
This is well defined outside the copy of $\mathbb P^{N-k-1}$
given by $[0:0:0:0:x_k+1: \cdots x_N].$

\begin{lemma}\label{projection}
Let $g:X\rightarrow\mathbb P^N$ be any embedding, $N\geq 2n+2$.  
Then for any $\epsilon>0$ there exists an automorphism
$$
\Phi([x_{0}: \cdots :x_{N}])=[x_{0}+\sum_{j=2n+2}^N\epsilon_{0,j}x_{j}: \cdots :
x_{2n+1}+\sum_{j=2n+2}^N\epsilon_{2n+1,j}x_{j}:x_{2n+2}:\cdot\cdot\cdot:x_{N}],
$$
such that $|\epsilon_{i,j}|<\epsilon$ for all $i,j$ and $\pi_{2n+1}\circ\Phi\circ g$
is an embedding.  In fact this holds for almost all choices of $\{\epsilon_{i,j}\}$.  
\end{lemma}
\begin{proof}
First let $H_{N}$ denote the hyperplane $\{x_{N}=0\}$.  If $q\in\mathbb P^N$
is a point $q\notin H_{N}$ there is a projection  $\pi_q$ from $\mathbb P^N\setminus\{q\}$
to $H_{N}$ defined as follows.  For a point $p\in\mathbb P^N\setminus\{q\}$ there is a unique
line passing through $p$ and $q$.  This line will intersect the hyperplane $H_{N}$
at a single point, and we let this point be $\pi_{q}(p)$.  If $q=[0:\cdot\cdot\cdot:1]$, 
this projection is the map 
$[x_{0}:\cdot\cdot\cdot:x_{N}]\mapsto [x_{0}:\cdot\cdot\cdot:x_{N-1}:0]$.
The lemma then follows from the fact that 
\begin{itemize}
\item[(a)] For almost all choices of $q$ the map $\pi_{N-1}\circ g$
is an embedding as long as $N\geq 2n+2$. 
\end{itemize}
Before we justify this standard fact we show how the lemma follows. 
Let $q$ be given by $q=[\delta_{0}:\cdot\cdot\cdot:\delta_{N}]$.
Then $\pi_{q}$ is given by 
\begin{itemize}
\item[(b)]
$
[x_{0}:\cdot\cdot\cdot:x_{N}]\overset{\pi_{q}}{\mapsto}[x_{0}-\frac{\delta_{0}x_{N}}
{\delta_{N}}:\cdot\cdot\cdot:x_{N-1}-\frac{\delta_{N-1}x_{N}}
{\delta_{N}}:0]
$
\end{itemize}
Since we can choose $q$ arbitrarily close to the point $[0:\cdot\cdot\cdot:0:1]$, 
we see that the lemma follows by repeated use of (a) and (b).  We can wiggle $\delta$ to make sure that the components of the new embedding do not have a common zero.

\

To show (a) let $Z$ be the subset of $W:=\mathbb P^N\times\mathbb P^N\times\mathbb P^N$
defined by $p\in g(X), p'\in g(X), p\neq p'$, and $r$ is on the (unique) line that contains $p$ and $p'$.
If $p=p'\in g(X),$ we let $r$ be on the tangent line to $g(X)$ at $p.$
One sees that $Z$ is a  2n+1 dimensional algebraic set in $W$.  Hence the image of $Z$ under the projection 
$(p,q,r)\mapsto r$ has zero measure in $\mathbb P^N.$ Hence, almost all choices of $q\in \mathbb P^N$ gives an embedding.

\end{proof}

\subsection{Projective Limits of Riemann Surfaces}

\subsubsection{Ergodic Properties}

\

In the case of Riemann surface laminations one can always 
construct directed currents as cluster points of push forwards
of currents on the universal cover of a leaf, \cite{FS2005, FS2006,FS2008}.

\begin{thm}
Let $X_\infty$ be a projective lamination by Riemann surfaces. If a leaf $L$ is covered by the unit disc $D,\phi:D \rightarrow L,$ then
$\tau_r= \frac{1}{m_r} \phi_*(\log^+ (r/|\zeta|)D) \rightarrow T$ when $r \rightarrow 1.$ If a leaf is covered
by $\mathbb C$ then there is a subsequence $r_n \nearrow \infty$ such that $\frac{1}{m(r_n)}\phi_*(D_{r_n}) \rightarrow T $.
\end{thm}

\begin{proof}
The factor $\frac{1}{m(r)}$ is a normalization in order to have currents of mass $1.$ It is shown in \cite{FS2005} that all cluster points of the family $\tau_r$ are positive and $\partial \overline{\partial}$-closed. The uniqueness in Theorem 1 implies that they are equal to $T.$ The case of $\mathbb C$ is classical and follows from the length over area estimates due to Ahlfors, see for example \cite{FS2008} p. 417.
\end{proof}

\subsection{Holomorphic Maps}

For a projective lamination over a torus we can construct 
holomorphic self maps.   \

\begin{prop}\label{selfmap2}
Let $X_{\infty}$ be a projective limit over a torus $X_{1}$.  Then for any two 
points $\{x_{n}\}$ and $\{x_{n}'\}$ in $X_{\infty}$ there exists a 
holomorphic automorphism $\varphi:X_{\infty}\rightarrow X_{\infty}$  
such that $\varphi(\{x_{n}\})=\{x_{n}'\}$.  
\end{prop}
We will use the following lemma, the proof of which is left to the reader. 
\begin{lemma}\label{selfmap1}
Let $X_\infty=(X_n,f_n)$ and $Y_\infty=(Y_n,g_n)$ be projective laminations with 
leaf dimension equal to $k.$ Let $\sigma:X_{\infty}\rightarrow Y_{\infty}$
be a map, 
 $\sigma\{x_n\}=\{\sigma_n(x_n)\}$ be a map from $X_\infty$ to $Y_\infty$. Then $\sigma$ is well defined if and only if $\sigma_n \circ f_n=g_n\circ \sigma_{n+1}.$ When $\sigma_n:X_n \rightarrow Y_n$ are holomorphic, then $\sigma$ is a continuous map holomorphic on leaves.
\end{lemma}

\emph{Proof of Proposition \ref{selfmap2}:} By lifting everything to the complex plane 
we will represent the sequence of coverings by linear maps $f_{n}:\mathbb C\rightarrow\mathbb C$.
Define $\sigma_{n}(\zeta):=\zeta+(\zeta_{n}'-\zeta_{n})$, where $\zeta_{n}$ and 
$\zeta_{n}'$ are representatives for the classes of $x_{n}$ and $x_{n}'$ respectively.
Note that $f_{n}(\zeta_{n+1})-\zeta_{n}$ and $f_{n}(\zeta_{n+1}')-\zeta_{n}'$
are in the class of zero.  Then 
\bea
\sigma_{n}\circ f_{n}(\zeta) & = & f_{n}(\zeta) + \zeta_{n}' - \zeta_{n}\\
& \sim & f_{n}(\zeta) + f_{n}(\zeta_{n+1}') - f_{n}(\zeta_{n+1}) \\
& = & f_{n}(\zeta + \zeta_{n+1}' - \zeta_{n+1}) \\
& = & f_{n}\circ\sigma_{n+1}(\zeta), 
\eea
and the proposition follows from Lemma \ref{selfmap1}. 
$\hfill\square$

\begin{remark}
In general it is not possible to construct maps $\sigma_{n}$ such that 
$\sigma_{n}(x_{n})=x_{n}'$.  Starting with a surface 
$X_{1}$ of genus 2 it is possible to construct coverings $\rightarrow X_{3}\rightarrow X_{2}\rightarrow X_{1}$
with a point $x_{1}\in X_{1}$, preimages $x_{3}$ and $x_{3}'$ in $X_{3}$ without 
a Deck transformation for the covering $X_{3}\rightarrow X_{1}$ such that 
$x_{3}$ is sent to $x_{3}'$.  So we have to assume that all the coverings of $X_{1}$ are Galois coverings, i.e. the deck transformations corresponding to $X_n\rightarrow X_1$ are transitive on fibers.
This is a sufficient condition. 
\end{remark}

\subsubsection{Projective limits over Riemann surfaces and topology of leaves}

Let $X$ be a compact Riemann surface of genus $g_{X}\geq 1$.  
We want to construct 2 to 1 unbranched covers of $X$.  
Let $\Gamma=\{\gamma_{1},...,\gamma_{k}\}$ be a collection of pairwise disjoint simple closed 
curves such that $X\setminus\Gamma$ is connected.  Since $X$ is orientable each 
curve $\gamma_{j}$ has a neighborhood $U_{j}$ such that $U_{j}\setminus\gamma_{j}$
has two connected components $U_{j}^1$ and $U_{j}^2$.  We let $\gamma_{j}^i$
denote the curve $\gamma_{j}$ depending on which connected component you approach it 
from, i.e., we can think of $X\setminus\Gamma$ as a bordered Riemann surface with $2$k 
boundary components.    \

Let $X_{1}$ and $X_{2}$ be two exact copies of the Riemann surface $X$.  
We get a new Riemann surface $Y$ by identifying each curve $\gamma_{j}^1$ in 
$X_{1}$ with $\gamma_{j}^2$ in $X_{2}$  and vice versa.  Because $X\setminus\Gamma$
is connected, we see that $Y$ is connected.  There is a natural projection $\pi:Y\rightarrow X$,
and this projection is clearly a 2 to 1 unbranched covering map.  We call $Y$
\emph{the 2 to 1 unbranched covering with respect to $\Gamma$}. \

In the case when $X_1$ is a Riemann surface we have a lot of flexibility to control the 
topology of leaves. We first prove a variation of Sullivan's example 6,  \cite{S1976}.

\begin{prop}
When $X_1$ is of genus $1$, all the leaves in $X_\infty$ are of the same nature, $\mathbb C$ or
$\mathbb C^*.$
\end{prop}

\begin{proof}
Let $(a_n)$ and $(b_n+ic_n)$ generate the lattice for $X_n$. The dichotomy is determined by either both $a_n$ or $(b_n+ic_n)$ go to infinity or one of the sequences is uniformly bounded.
\end{proof}

In \cite{S1991} Sullivan has introduced the Teichm\"{u}ller space for hyperbolic 
laminations, i.e., all leaves are hyperbolic.  Here we just show that for the projective 
laminations with fibre $\mathbb C^*$ (resp. $\mathbb C$) we have infinitely many inequivalent realizations. 
Let us consider projective limits over tori defined by lattices $\Gamma_{\beta}$
generated by $\{1,\beta i\}, \beta\in\mathbb R$ over $\mathbb Z$.  We let $\Gamma_{\beta,j}$
denote the lattice $2^j\Gamma_{\beta}$.  The maps $f_{j}:\mathbb C
\overset{\mathrm{incl}}{\rightarrow}\mathbb C$ induce 4-1 unbranched covers 
$\tilde f_{j}:\mathbb C/\Gamma_{\beta,j+1}\rightarrow\mathbb C/\Gamma_{\beta,j}$.
A point in the projective limit is represented by a sequence $\{\zeta_{j}\}_{j=1}^\infty$,
$\zeta_{j}\in\mathbb C$, where $\zeta_{j+1}-\zeta_{j}\in\Gamma_{\beta,j}$ 
and $\zeta_{j}\sim\zeta_{j}'$ if $\zeta_{j}-\zeta_{j}'\in\Gamma_{\beta,j}$. 
We denote this projective limit by $X_{\beta}$.

Let us parametrize the leaf containing the point $\{0\}$; let $\varphi_{\beta}:\mathbb C\rightarrow X_{\beta}$
be defined by $\zeta\mapsto\{\zeta\}$.   Since a point $\zeta$ can be in $\Gamma_{\beta,j}$
for all $j$ only if $\zeta$ is zero we see that $\varphi_{\beta}$ is injective.  We also 
have that for any point $\zeta\in\mathbb C$
\begin{itemize}
\item[(a)] $\varphi_{\beta}(\zeta+2^j)\rightarrow\varphi_{\beta}(\zeta)$ and 
\item[(b)]$\varphi_{\beta}(\zeta+2^j\beta i)\rightarrow\varphi_{\beta}(\zeta)$
\end{itemize}
uniformly as $j\rightarrow\infty$. 

\begin{prop}
If $X_{\beta_{1}}$ and $X_{\beta_{2}}$ are isomorphic then $\beta_{2}=r\beta_{1}$
with $r\in\mathbb Q$. 
\end{prop}
\begin{proof}
Assume that there exists an isomorphism $\psi:X_{\beta_{1}}\rightarrow X_{\beta_{2}}$.
Since projective limits over tori are homogenous we may assume that the point $\{0\}_{1}$
is taken to the point $\{0\}_{2}$.  Now let $\tilde\psi$ denote the map $\tilde\psi:=\varphi_{\beta_{2}}^{-1}\circ\psi\circ\varphi_{\beta_{1}}$.
By compactness and (a) and (b) we get that 
\begin{itemize}
\item[(c)] $\tilde\psi(\zeta+2^j)\rightarrow\tilde\psi(\zeta)$ and 
\item[(d)]$\tilde\psi(\zeta+2^j\beta_{1} i)\rightarrow\tilde\psi(\zeta)$
\end{itemize}
uniformly mod $\Gamma_{\beta_{1}}$ as $j\rightarrow\infty$.  Write $\tilde\psi(\zeta)=\lambda\cdot\zeta$
where $\lambda=x+iy$.  Then (c) implies that $\lambda 2^j=x2^j+iy2^j\rightarrow 0$
mod $\Gamma_{\beta_{2}}$.  Then by Lemma \ref{mod} below we have that 
$x=2^{-N_{1}}k_{1}$ and $y=2^{-N_{2}}k_{2}\beta_{2}$.  Likewise, from (c)
and the same lemma we get that $y=\frac{2^{-N_{3}}k_{3}}{\beta_{1}}$
and $x=\frac{2^{-N_{4}}k_{4}\beta_{2}}{\beta_{1}}$.  
\end{proof}

\begin{lemma}\label{mod}
Let $\alpha\in\mathbb R$ and let $y\in\mathbb R$ such that $2^j\cdot y\rightarrow\alpha$
mod $\alpha\mathbb Z$ uniformly as $j\rightarrow\infty$.  Then $y=2^{-N}k\alpha$
for $k,N\in\mathbb Z$.
\end{lemma}
\begin{proof}
Let $\epsilon>0$ satisfy $2\epsilon<\frac{|\alpha|}{2}$.  By assumption there exists 
an $N$ such that the equation 
\begin{itemize}
\item[$(*)$] $2^j\cdot y=k\alpha+\delta$ 
\end{itemize}
has a solution with $k\in\mathbb Z$ and $|\delta|<\epsilon$ for all $j\geq N$.
Assume to get a contradiction that $k,\delta$ solves $(*)$ with $j=N$ and $\delta\neq 0$.
Let $l\in\mathbb N$ be the smallest integer such that $2^l|\delta|\geq\epsilon$.
Then $\epsilon\leq 2^l|\delta|<2\epsilon$.   Then $\mathrm{dist}(2^{N+l}y,\alpha\mathbb Z)\geq\epsilon$
which is a contradiction.
\end{proof}

\medskip

We now proceed to study the topology of leaves. 
\begin{prop}
Let $X_{1}$ be a compact Riemann surface of genus $g_{X_{1}}\geq 1$.
Then there exists a projective limit $X_{\infty}$ with $\pi_{1}:X_{\infty}\rightarrow X_{1}$ 
such that each leaf in the lamination is simply connected. 
\end{prop}

Before we start we give a brief sketch of the proof, whose details are given in 
Lemma 5 through Lemma 12 below.  Let $\gamma$ be a simple closed 
curve in $X_{1}$. Through the cutting procedure for 
making holomorphic covers it is possible to construct a composition of covers 
of $X_{1}$ such that any lifting of $\gamma$ is an open curve.  First we construct a 
cover such that for any lift $\tilde\gamma$ of $\gamma$, the complement is connected (Lemma 5).
Then we construct an additional cover such that any lift is an open curve (Lemma 10).    
Through an inductive procedure we run through all elements of the fundamental 
group of $X_{1}$ to make sure that any element eventually lifts to an open curve; this 
implies that all leaves are simply connected.  The main difficulty lies in the fact that 
for any cover we construct, we construct many new closed lifts of the curves we did not open yet.

\begin{proof}
Choose a base point $y\in X_{1}$ and choose a basis $\{\gamma_{j}\}_{j=1}^{2g}$
for $\pi_{1}(X_{1})$
consisting of closed smooth loops $\gamma_{j}:[0,1]\rightarrow X_{1}$, 
$\gamma_{j}(0)=\gamma_{j}(1)=y$.  The idea is to construct a sequence of coverings 
such any representative of an element in $\pi_{1}(X_{1})$ eventually lifts 
to an open curve in $X_{k}$.  This is achieved through  
Lemma \ref{mainlift} below, according to which there exists 
a tower $\{X_j\}_{j=1}^\infty$ such that for any finite non-trivial
composition $\gamma=\gamma_{j_{s}}\cdot\cdot\cdot\gamma_{j_{1}}$, there exists
a $k\in\mathbb N$ such that any lift $\tilde\gamma$ 
of $\gamma$ to $X_{k}$ is a non closed curve, $\tilde\gamma(0) \neq\tilde\gamma(1)$.  \

Let $X_{\infty}^0$ be a leaf  and choose a base point 
$x\in X_{\infty}^0$ with $\pi(x)=y$.  Assume to get a contradiction 
that there exists a non-trivial curve $\gamma:[0,1]\rightarrow X_{\infty}^0$,
$\gamma(0)=\gamma(1)=x$.   Now $\pi_{1}(\gamma)$ is homotopic to a 
finite composition $\gamma_{1}=\gamma_{j_{s}}\cdot\cdot\cdot\gamma_{j_{1}}$,
hence $\pi_{1}(\gamma)$ is homotopic to a curve that lifts to a non closed curve in some $X_{k}$.
But then $\gamma$ is homotopic to a non closed curve, which is not possible. 

\end{proof}

\begin{lemma}\label{lift1}
Let $\lambda$ be a simple closed curve in a compact Riemann surface $X$ such that 
the complement of $\lambda$ is disconnected, i.e., $X\setminus\lambda$ has 
two connected components $U_1$ and $U_2$.  Let $\gamma$ be a simple closed curve
in $X$ that intersects $\lambda$ transversally in exactly two points, and with the property
that $U_j\setminus\gamma$ is connected for $j=1,2$.  Let $Y$ be the 2 to 1 unbranched 
covering of $X$ with respect to $\gamma$.  Then $\lambda$ lifts to two disjoint curves $\lambda^l_1$
and $\lambda^l_2$ and each of them has connected complement in $Y$.  
\end{lemma}

\begin{proof}
Note that since $U_j \setminus \gamma$ is connected for $j=1,2$, it follows that
$X \setminus \gamma$ is connected. Let $X_1,X_2$ be two copies of $X$ as in the 
construction of the natural $2$-sheeted cover of $X$ with respect to $\gamma.$ We consider a lifting of $\lambda$. Note that when we follow the lifting and cross
a copy of $\gamma$, we switch between $X_1$ and $X_2$. When we continue the lifting of the curve we will cross a copy of $\gamma$ for the second time before we have moved all around $\lambda.$ So this puts us back in the original copy of $X.$ In fact we see that because of this the lifting of $\lambda$ consists of two disjoint copies of $\lambda.$ 

Since $U_1\setminus \gamma$ is connected, we can connect any two points in $\pi^{-1}(U_1)$
with a curve in $\pi^{-1}(U_1)$. The same goes for $U_2.$ Next pick two points in $Y$ close to each other but locally on opposite sides of $\lambda^l_1.$ We connect those by curves in $\pi^{-1}(U_1\cup U_2)$ to points close to each other on opposite sides of $\lambda^l_2.$ Then we connect them by crossing over $\lambda^l_2.$ We can do the same to show that the complement of $\lambda^l_2$ is also connected.
\end{proof}

\begin{lemma}\label{lift2}
Let $X$ be a compact Riemann surface. Suppose that $\{\lambda_j\}_{j=1}^k$ is a finite family of simple closed noncontractible curves. Assume that they are pairwise disjoint. Assume moreover that
each of them disconnects the complement into two open sets, $U_j^1,U_j^2.$ Furthermore assume that
$U_j^1$ contains none of the other curves.  Then there exists a  
2-1 unbranched covering $Y\rightarrow X$ such 
that 
if $\lambda^l_j$ is a lifting of $\lambda_j$ to $Y$,  then $Y\setminus\lambda_j^l$ is connected.  
\end{lemma}

\begin{proof}
We first do the proof for one curve, then we do it for any number $k \ge 2$. \

Let $U_1$ and $U_2$ denote the two connected components of $X\setminus\lambda_1$, and 
let $x\in\lambda_1$ be a smooth point.   We want to show that "there is a handle in each $U_j$".
Let $f_j:U_j\rightarrow X_j$ be conformal maps for $j=1,2$, such 
that $X_j\setminus f_j(U_j)$ is a smoothly bounded simply connected domain in a compact Riemann surface $X_j$. 
Let $x_1$ and $x_2$ be the points $f_1(x)$ and $f_2(x)$ repectively.  
Since $\lambda_1$ is non-contractible we see that neither $X_1$ nor $X_2$ can be the 
Riemann sphere, and so there exist simple closed curves $\gamma_j:[0,1]\rightarrow f(U_j)\cup\{x_j\}$
, $\gamma_j(0)=\gamma_j(1)=x_j$,
such that $\gamma_j$ does not separate $f(U_j)$ for $j=1,2.$  We may then separate the pullback union $f_1^{-1}(\gamma_1)\cup
f_1^{-1}(\gamma_2)$ at the point $x$ to get a curve $\gamma^1$ such that $\gamma^1$ intersects 
$\lambda_1$ transversally and exactly at two points.   Then by Lemma \ref{lift1} the conclusion follows. 

Next we assume $k \geq 2.$

Pick two distinct points $a_j,b_j$ in $\lambda_j.$ The set $X'=X \setminus \cup  U_j^1$ is a bordered
connected Riemann surface. We connect $b_j$ to $a_{j+1}, j<k$. We can do this inductively with disjoint curves $\sigma_j$ without them intersecting each other, without self intersections and without destryoing
connectedness of the complement of the curves. We also add a curve $\sigma_k$ from $b_k$ to $a_1.$   Next we observe that if we replace any $U_j^2$ by a disc we get a  compact Riemann surface which is not a sphere.  We connect $a_j$ to $b_j$ inside $U_j^1$ without disconnecting $U_j^1.$ Putting all the $2k$ curves together we get a simple closed curve $\gamma$ with connected complement which cuts each
$\lambda_j$ in exactly two points. Moreover each $U_j^i \setminus \gamma$ is connected. Hence we can apply Lemma \ref{lift1}. 
\end{proof}

\begin{lemma}
Let $Y\overset{f}{\rightarrow} X$ be a 2-1 unbranched covering. Let $\lambda$ be a closed curve in $X.$ Assume that $X \setminus \lambda$ is connected. Let $\lambda_1,\lambda_2$ be the lifts of $\lambda$ (possibly $\lambda_1=\lambda_2$). Then $Y \setminus \lambda_j$ is connected.
\end{lemma}
\begin{proof}
Let $\tilde\lambda$ be a lifting of $\lambda$ and let $\tilde a,\tilde b\in Y\setminus\tilde\lambda$
be two points closed to each other, locally on opposite sides of $\tilde\lambda$.  Let $a=f(\tilde a), b=f(\tilde b)$. 
By assumption there is a curve $\gamma\in X\setminus\lambda$ connecting $a$ and $b$.  Let $\tilde\gamma$
denote the lifting of $\gamma$ with initial point $\tilde a$.  If the end point of $\tilde\gamma$ 
is $\tilde b$ we are done.    Otherwise there are two cases to consider. 

Assume first that $\lambda$ lifts to two disjoint closed curves.  In that case consider the 
loop $\gamma'$ which is obtained by extending $\gamma$ crossing over $\lambda$ 
back to the initial point $a$.  The lifting of $\gamma'$ will not intersect $\tilde\lambda$, 
but the other pre-image.  Thus lifting the composition $\gamma\cdot\gamma'$ gives 
a curve connecting $a$ and $b$ that does not intersect $\lambda$.  

The other possibility is that $\lambda$ lifts to an open curve, \emph{i.e.}, we have to 
traverse $\lambda$ twice to obtained a lifted closed curve.  In that case let $\mu$
denote a loop based at the point $b$ homotopic to $\lambda$ with $\mu\cap\lambda=\emptyset$. 
Then the lift of $\mu\cdot\gamma$ will connect $\tilde a$ and $\tilde b$. 
\end{proof}

\begin{lemma}\label{lift2}
Let $X$ be a compact Riemann surface. Suppose that $\{\lambda_j\}_{j=1}^k$ is a finite 
family of simple closed pairwise disjoint noncontractible curves. 
We assume they each disconnect $X.$ Then there exists a composition
$$
Y_s\rightarrow Y_{s-1}\rightarrow\cdot\cdot\cdot\rightarrow Y_1\rightarrow X
$$
of 2-1 unbranched coverings such 
that 
if $\lambda^s_j$ is a lifting of $\lambda$ to $Y_{s}$, then either $\lambda_j^s$ is an open curve or
$Y_{s}\setminus\lambda_j^s$ is connected.  
\end{lemma}

\begin{proof}
We prove it by induction. Observe at first that if $k \leq 2,$ we are automatically in the situation of 
Lemma \ref{lift2}. Assume the lemma holds for some $k.$ Suppose that we have curves $\{\lambda_j\}_{j\leq k+1}$ as in the statement. Note that if two of the curves are homotopic it suffices to do $k$ of them so we are done.
Hence we assume no two of them are homotopic.
We see easily by induction on the number of curves that for at least one of them, say $\lambda_{k+1}$
all the other curves are in the same connected component, say $U^2_{k+1}$ of the complement.
We replace $U_{k+1}^1$ by a disc $\Delta$ and get a compact Riemann surface $X'$ still containing the curves $\lambda_1,\dots,\lambda_k.$ By assumption none of them are homotopic to $\lambda_{k+1}$ hence none of them are contractible in $X'.$ Hence we can use the inductive hypothesis and find a Riemann surface $Y_s$ so that all liftings of $\lambda_1,\dots,\lambda_k$ satisfy the conditions of the Lemma.  Next we replace $Y_s$ by $Y_s'$ by replacing each 
copy of $\Delta$ by a copy of $U_{k+1}^1.$
Then we are in the situation of Lemma \ref{lift2} where the curves are all the liftings to $Y'_s$ of $\lambda_{k+1}.$ 
\end{proof}

\begin{lemma}\label{lift3}
Let $X$ be a compact Riemann surface. 
Suppose that $\{\lambda_j\}_{j=1}^k$ is a finite family of simple closed noncontractible curves. Then there exists a composition
$$
Y_s\rightarrow Y_{s-1}\rightarrow\cdot\cdot\cdot\rightarrow Y_1\rightarrow X
$$
of 2-1 unbranched coverings such 
that 
if $\lambda^s_j$ is a lifting of $\lambda$ to $Y_{s}$, then either $\lambda_j^s$ is an open curve or
$Y_{s}\setminus\lambda_j^s$ is connected.  
\end{lemma}

\begin{proof}

We will prove this by induction on the number $k$, and so we start by proving it for $k=1$.
But this case is covered by Lemma 5.

Assume now that the lemma holds for $k$ curves, $k\geq 1$, and that we are given $k+1$ curves.
We start by choosing a sequence of liftings 
$$
Y_s\rightarrow Y_{s-1}\rightarrow\cdot\cdot\cdot\rightarrow Y_1\rightarrow X  
$$
such that the conclusion holds for $\lambda_1,...,\lambda_k$.   Consider the liftings of $\lambda_{k+1}$
to $Y_s$.   We ignore the liftings that are either open curves or has connected complements, and so we
are left with liftings $\lambda_{1}^{s},...,\lambda_{m}^s$.  These curves are pairwise disjoint.  
Hence Lemma \ref{lift2} applies and we are done.
\end{proof}

\begin{lemma}\label{open1}
Let $Y_s\rightarrow Y_{s-1}\rightarrow\cdot\cdot\cdot\rightarrow Y_1\rightarrow X$
be a sequence of 2 to 1 unbranched coverings of a compact Riemann surface $X$, and 
let $\lambda$ be a simple non-contractible curve in $X$ that does not separate $X$.
Then there exists a 2 to 1 covering $Y_{s+1}\rightarrow Y_{s}$ such that 
all liftings of $\lambda$ to $Y_{s+1}$ are open curves.  
\end{lemma}
\begin{proof}
Let $y\in X$ be a point on $\lambda$, so that we may parametrize $\lambda$
by a map $\lambda:[0,1]\rightarrow X$ such that $\lambda(0)=\lambda(1)=y$. \

Let $\gamma:[0,1]\rightarrow X$ be a simple closed curve, $\gamma(0)=\gamma(1)=y$,
such that $\gamma$ intersects $\lambda$ transversally and only at one point.  Such a 
curve exists by the assumption made on $\lambda$.
The complement
of $\gamma$ is then connected. \

Now consider liftings of (multiples of) $\gamma$ to $Y_{s}$.  
Fix a base point $y^l_{1}\in Y_{s}$ such that $\pi(y^l_{1})=y$, $\pi:=\pi_{1}\circ\cdot\cdot\cdot\circ\pi_{s-1}$.
There is a smallest integer $n_{1}$ such that the lifting of $n_{1}\gamma$ with base point
$y^l_{1}$ is a closed curve, but all liftings of $m\gamma$, $m<n_{1}$, with base point 
$y^l_{1}$ are open curves (or $n_{1}=1$).  Let $\gamma^l_{1}$ denote the lifting of $n_{1}\gamma$ with 
base point $y^l_{1}$.  \

The closed curve $\gamma^l_{1}$ passes through $n_{1}$ different pre-images of 
the point $y$.  If there are some pre-images left, choose a $y^l_{2}$ with 
$\pi(y^l_{2})=y$ such that $y^l_{2}$ does not lie in $\gamma^l_{1}$.
Repeating the argument above, one finds a simple closed curve $\gamma_{2}^l$
which is the lifting of $n_{2}\gamma$ with base point $y^l_{2}$.   Then 
$\gamma^l_{2}$ does not intersect $\gamma^l_{1}$.  \

Repeating this finitely many times we end up with a set of pairwise disjoint closed curves
$\{\gamma_{j}^l\}_{j=1}^k$, where $\gamma^l_{j}$ is a lifting of $n_{j}\gamma$, 
such that any point in $\pi^{-1}(y)$ lies in a curve $\gamma^l_{j}$.  \

Now let $y^l_{1},...,y^l_{m}$ be all the pre-images of $y$ such that 
the lifts $\lambda_{j}^l$ of $\lambda$ determined by the base point $y^l_{j}$
is a closed curve.  We have to choose the cover $Y_{s+1}\rightarrow Y_{s}$ such 
that all liftings of these curves are open. \

For each $y^{l}_{j}$ choose the closed curve $\gamma_{j}^l$ that contains this point. 
We let $\Gamma$ denote the collection of all these closed curves. Then for each 
curve $\lambda^l_{j}$ there is precisely one curve in $\Gamma$ that intersects it, and the intersection 
is transversal and at exactly one point.  For each curve in $\Gamma$ there is at least
on curve $\lambda^l_{j}$ that intersects it.  This means that $Y_{s}\setminus\Gamma$
is connected, since each side of each curve in $\Gamma$ is in the same connected component
of $Y_{s}\setminus\Gamma$ - some $\lambda^l_{j}$ is a path from one side of the curve to the other.  \

So let $Y_{s+1}$ be the unbranched 2 to 1 cover with respect to $\Gamma$ and we are done.        
\end{proof}

\begin{lemma}\label{open2}
Let $X$ be a compact Riemann surface and let $\{\lambda_j\}_{j=1}^k$ denote a 
finite family of simple closed curves which are not contractible. 
We can then find a sequence 
$$
Y_s\rightarrow Y_{s-1}\rightarrow\cdot\cdot\cdot\rightarrow Y_1\rightarrow X
$$
of 2-1 unbranched coverings such 
that if $\gamma_j^s$ is a lifting of $\gamma_j$ to $Y_s$, then $\gamma_j^s$ is an open curve.  
\end{lemma}

\begin{proof}
By Lemma \ref{lift3} we can assume that all the $\lambda_j$ have connected complement.
Again we will prove it by induction on $k$.  If $k=1$ we choose a simple closed curve $\gamma$
that intersects $\lambda$ transversally at exactly one point.  Then the 2 to 1 unbranched covering of $X$
with respect to $\gamma$ will do.  Assume that the lemma holds for some $k\geq 1$
and that we are given $k+1$ curves.   Then there is a sequence
$$
Y_s\rightarrow Y_{s-1}\rightarrow\cdot\cdot\cdot\rightarrow Y_1\rightarrow X
$$
such that all liftings of the curves $\lambda_1,...,\lambda_k$ are non closed curves.  By Lemma 
\ref{open1} there exists a $Y_{s+1}$ such all liftings of $\lambda_{k+1}$ to $Y_{s+1}$
are open curves.   
\
\end{proof}

\begin{lemma}\label{mainlift}
Let $X_{1}$ be a compact Riemann surface of genus $g_{X_{1}}\geq 1$, 
let $y\in X_{1}$ be a point, and let 
$\Gamma=\{\gamma_{j}\}_{j=1}^{2g}$ be smooth loops $\gamma_{j}:[0,1]\rightarrow X$,
$\gamma_{j}(0)=\gamma_{j}(1)=y$, such that $\Gamma$ is a basis for the homotopy of $X$.  
There exists a tower $\{(X_{j},\pi_{j})\}_{j=1}^\infty$ of 2 to 1 unbranched 
coverings such that for any non-trivial finite composition 
$\gamma=\gamma_{j_{s}}\cdot\cdot\cdot\gamma_{j_{1}}$, there exists a
$k\in\mathbb N$ such that any lifting of $\gamma$ to $X_{k}$ is an open curve. 
\end{lemma}

\begin{proof}
We will prove this by induction on the length of the compositions.  First we apply 
Lemma \ref{open2} to get a sequence
$$
X_{k_{1}}\rightarrow\cdot\cdot\cdot\rightarrow X_{1}
$$
such that any lift of any of the basis elements $\gamma_{j}$ is an open curve. \

Now assume that we have a sequence 
$$
X_{k_{s}}\rightarrow\cdot\cdot\cdot\rightarrow X_{k_{1}}\rightarrow\cdot\cdot\cdot\rightarrow X_{1}
$$
such that any lifting of a curve $\gamma=\gamma_{j_{1}}\cdot\cdot\cdot\gamma_{j_{l}}$
with $l\leq s$ to $X_{k_{s}}$ is an open curve.  Then any lift of 
a curve $\gamma_{0}=\gamma_{j_{1}}\cdot\cdot\cdot\gamma_{j_{s+1}}$
is either an open curve or a simple smooth closed curve (since the composition of
the $s$ first components give an open curve).  Let $\lambda_{1},...,\lambda_{m}$
denote the finite number of such closed lifts to $X_{k_{s}}$.  By Lemma \ref{open2}
there exists a sequence 
$$
X_{k_{s+1}}\rightarrow\cdot\cdot\cdot\rightarrow X_{k_{s}}
$$
such that all liftings of these curves to $X_{s+1}$ are open curves. 
\end{proof}

\medskip

\begin{remark}
The construction is flexible enough so that $X_\infty'$ can be made pluripolar or can be made of tranverse dimension zero.
\end{remark}

It is not hard to modify the above construction so that the topologies of the leaves
vary. 

\begin{prop}
There exists a projective limit $X_\infty$ such that all leaves 
except one are simply connected.  
\end{prop}
The idea is to repeat the above construction, but to make sure that there exists 
one sequence $\{y_{j}\}$ of preimages of a point $y_{1}\in Y$ and a closed curve 
$\gamma_{1}$ based at $y_{1}$ that never lifts to an open curve.  We make sure that we open 
all curves that will not represent curves in the leaf containing the point $\{y_{j}\}$.
\begin{proof}
Let $X_{1}$ be a compact Riemann surface of genus greater than or equal to 2 and 
let $\Omega_{1}\subset X_{1}$ be a smoothly bounded domain which is homeomorphic
to a torus with a disk removed.  The boundary $\partial\Omega_{1}$ separates $X_{1}$.   
Fix a point $x_{1}\in X_{1}\setminus\overline{\Omega_{1}}$ as a base point 
for the fundamental group of $X_{1}$, and fix a set of generators $\Gamma=\{\gamma_{1},...,\gamma_{2g_{X_{1}}}\}$
of $\pi_{1}(X_{1},x_{1})$. 

We will construct a sequence of unbranched coverings 
by induction, so assume that we have constructed a sequence 
$$
X_{k}\longrightarrow\cdot\cdot\cdot\longrightarrow X_{1}
$$ 
for $k\geq 1$, and that we in each $X_{j}$ have fixed a pre-image $\Omega_{j}$
of $\Omega_{1}$ which is mapped homeomorphically onto $\Omega_{1}$ by the projection 
onto $X_{1}$.  The boundaries $\partial\Omega_{j}$ separate the $X_{j}$'s.

The cover $X_{k+1}\rightarrow X_{k}$ is constructed as follows.  First let 
$\Gamma_{k}$ denote the set of lifts to $X_{k}$ of curves in $\Gamma$ to 
simple closed curves.  There are only a finite number of such curves.  Let $\tilde\Gamma_{k}$
be the set of curves in  $\Gamma_{k}$ that are homotopic to curves in $X_{k}\setminus\Omega_{k}$.
Replace the representatives in $\tilde\Gamma_{k}$ by representative that are contained in 
$X_{k}\setminus\Omega_{k}$.  
Next replace 
$\Omega_{k}$ by a disk $D$ to obtain a compact Riemann surface $\tilde X_{k}$.  
Now let $\tilde\Gamma_{k}'$ be the set of all curves in $\tilde\Gamma_{k}$
that are still non-trivial in $\tilde X_{k}$.
By Lemma 
\ref{open2} there is a cover $\tilde X_{k+1}\rightarrow\tilde X_{k}$ 
such that all curves in $\tilde\Gamma_{k}'$ 
lift to open curves.  Finally replace all lifts of $D$ in $\tilde X_{k+1}$
by a copy of $\Omega_{k}$ to obtain the compact surface $X_{k+1}$ and choose
one of the copies of $\Omega_{k}$ and label it $\Omega_{k+1}$.

Now let $X_{j+1}\rightarrow X_{j}$ be the sequence of covers constructed 
according to the above procedure and let $X_{\infty}$ be the projective limit.
Let $y_{1}\in\Omega_{1}$ be a point and let $\{y_{j}\}$ be the point 
where $y_{j}$ is the pre-image of $y_{1}$ contained in $\Omega_{j}$.  The leaf $L$
containing $\{y_{j}\}$ is clearly not simply connected.  We claim that all other leaves 
are simply connected.  To see this let $\{x_{j}\}$ be a point in some leaf $L'$.   If this leaf 
is not simply connected there is a an element $\gamma\in\pi_{1}(X_{1},x_{1})$
that lifts to a simple closed curve $\tilde\gamma_{j}$ 
based at $x_{j}$ for all $j$ big enough.  Assume to get a contradiction that $L'\neq L$.
Then the distance $d_{j}(x_{j},y_{j})$ tends to infinity and so the distance between 
$x_{j}$ and $\Omega_{j}$ tends to infinity.  Since each lift $\tilde\gamma_{j}$
has the same length we have that $\tilde\gamma_{j}\cap\Omega_{j}=\emptyset$
for all $j$ big enough.  Consider $X_{j}$ for a big enough $j$.  If $\tilde\gamma_{j}$
was non-trivial in $\tilde X_{j}$ (obtained by replacing $\Omega_{j}$ by a disk D)
then the lift to $X_{j+1}$ would be an open curve.  Hence $\tilde\gamma_{j}$
is contractible in $\tilde X_{j}$ and $\tilde\gamma_{j}$ is free homotopic to 
$\partial\Omega_{j}$.  So $\tilde\gamma_{j+1}$ is free homotopic to 
one of the pre-images of $\partial\Omega_{j}$.  It has to be free homotopic to 
$\partial\Omega_{j+1}$, otherwise $\tilde\gamma_{j+2}$ would be an open curve. 
But then the distance between $x_{j+1}$ and $y_{j+1}$ is the same as the distance between 
$x_{j}$ and $y_{j}$.  Since the argument can be repeated for all $j$ big enough this 
is a contradiction.          
\end{proof}

\begin{remark}
The method used above for controlling topology of leaves is quite flexible, and 
it is possible to construct many more examples.  We will indicate a small 
modification of the above result
and then leave it to the interested reader to construct more examples.    

We will pay more attention to the leaf which is not simply connected 
to control its topology.  In the above construction choose the base point $x_{1}$
on $\partial\Omega_{1}$ instead of in $X_{1}\setminus\overline\Omega_{1}$.
At the inductive step above consider also all lifts of curves $\gamma$ 
to open curves $\tilde\gamma$ that are not homotopic to curves in $\Omega_{j}$.
Such a curve has to wind around a handle which is not in $\Omega_{j}$.  So when 
we replace $\Omega_{j}$ with a disk we may replace $\tilde\gamma$ with 
a non-trivial curve in $\tilde X$ which is homotopic to $\tilde\gamma$ mod $\Omega_{j}$.
Make sure that each such curve lifts to an open curve.  This way any curve $\gamma\in\pi_{1}(X_{1},x_{1})$
that survives in the limit will be homotopic to a curve in $\Omega_{1}$ and so 
the leaf has the same fundamental group as $\Omega_{1}$.  The result is a projective 
limit with \emph{all leaves except one simply connected and the remaining one homeomorphic
to a once punctured torus.}
\end{remark}

We now consider the higher dimensional case. 

\begin{prop}
Let $X$ be a compact complex manifold of dimension $\ell.$ 
Assume $S \subset X$ is a real orientable hypersurface such that 
$X \setminus S$ is connected. There is a double cover $Y$ of $X$, $f:Y \rightarrow X$ and an orientable hypersurface $\Sigma\subset Y$ with connected complement.
\end{prop}

\begin{proof}
To construct $Y$ we use the same construction as for Riemann surfaces cutting along $S.$ Let $\Sigma$ be one of the components of $f^{-1}(S)$. The proof that $Y \setminus \Sigma$ is connected is the same as for Riemann surfaces.
\end{proof}

\begin{thm}
There are $(X_\infty,f_n)$ projective limits of compact complex manifolds of dimension $\ell\geq 1$ such that $f_n:X_{n+1} \rightarrow X_n$ is a double cover and $X_\infty$ has the structure of a lamination of dimension $\ell.$ Every $\partial \overline{\partial}$-closed current on $X_\infty$ is closed. There is a unique directed positive closed current of mass $1$ on $X_\infty.$
If $X_1$ is projective then $X_\infty$ is embeddable in $\mathbb P^{2l+1}$.
\end{thm}

\begin{proof}
We start with $X_1$, a compact complex manifold containing an orientable hypersurface with connected complement. This permits to construct the abstract lamination.  To prove
embeddability we apply Theorem 2.
\end{proof}

\section{Furstenberg foliations}

In the previous sections we have constructed laminations that have few closed currents; in this section we will construct a minimal lamination with uncountably 
many closed laminated extremal currents.  The construction uses an example of 
Furstenberg (see \cite{Fu1961} page 585) originally constructed as a counterexample 
to the existence of ergodic averages for smooth selfmaps of certain product spaces.  

\

We start by recalling how to construct a lamination by suspension. 
Let $S$ be a compact complex manifold of dimension $k$ and 
let $\tilde S\overset{\pi}{\rightarrow} S$ be its universal cover.  Then $\ S=\tilde S/\Gamma$
where $\Gamma$ is the group of deck transformations.  Let $M$ be a smooth 
manifold and let $F:\Gamma\rightarrow\mathrm{Diff}(M)$ be a representation.  
We get a group $\tilde\Gamma$ acting freely and properly discontinuously on 
the product $\tilde S\times M$ by defining 
$$
\tilde F(\tilde s,x):=(\gamma(\tilde s),F(\gamma)(x)).
$$  
Hence we may consider the (smooth) manifold $Y:=(\tilde S\times M)/\tilde\Gamma$.
If $U\subset\tilde S$ is an open set such that $\pi:U\rightarrow S\rightarrow Y$ is injective 
then $\tilde\pi:U\times M$ is injective.  Hence $Y$ can be given the structure 
of a k-dimensional lamination with global transversal $M$ and $Y$ is naturally 
a CR-fibration over $S$ with fibers $M$.  If $M$ is a complex manifold and the 
image of $\Gamma$ is holomorphic then $Y$ is naturally a complex manifold. 
When the group $F(\Gamma)$ is amenable, there is a measure on $M$ invariant by $F(\Gamma)$
and hence there is a closed current directed by the lamination.

\begin{example}
Let $\Gamma=\{m\gamma_{1}+n\gamma_{2}\}_{m,n\in\mathbb Z}\subset \tilde S:=\mathbb C$
be a lattice, \emph{i.e.}, the quotient $S$ is a torus.  Let $M=S^1$ and let $\varphi:M\rightarrow M$
be the map $z\mapsto az$ with $|a|=1$ not a root of unity.  Let $F(\gamma_{1}):=F(\gamma_{2}):=\varphi$.  We get an $S^1$-fibration over the torus all of whose leaves are 
isomorphic to the punctured complex 
plane and they are all dense.  Since Lebesgue measure $\mu$
on $S^1$ is invariant under $\varphi$ we get a closed laminated current $T$ on 
$Y$ by letting $\mu$ be the transversal measure.  On the other hand any 
closed current on $Y$ gives rise to an invariant measure on $S^1$, and since 
$\mu$ is the unique measure on $S^1$ invariant under $\varphi$ the current $T$
is unique.    
\end{example}

The goal in what follows is to use $T^2:=\{|\zeta_{1}|=|\zeta_{2}|=1\}\subset\mathbb C^2$ 
as fibers instead of $S^1$.  
Furstenberg has constructed a smooth minimal automorphism $\varphi:T^2\rightarrow
T^2$ without
\emph{unique} invariant measures and so we will get several closed currents.  
We will make the construction explicit and show that $\varphi$ extends to a holomorphic map on $\mathbb C^* \times \mathbb P^1.$
For the benefit of the reader 
we will show below that all orbits of $\varphi$ are dense in the torus, \emph{i.e.}, 
it gives us a minimal lamination.  The density of the orbits is not explicitly written in 
Furstenberg's article but seems widely known. 
Our lamination will be a compact CR-manifold in a holomorphic 
$(\mathbb C^*\times\mathbb P^1)$-fibration over a compact Riemann surface S. 

\

Furstenberg constructed a family $\{T_\alpha\}$ of automorphisms on 
the real torus of the simple form 
$$
T_{\alpha}(\zeta_{1},\zeta_{2}):=(\mathrm{e}^{2\pi i\alpha}\zeta_{1},
g_{\alpha}(\zeta_{1})\zeta_{2})
$$  
where $\alpha$ is irrational and $g$ is a $\mathcal C^\infty$-smooth function 
on the unit circle.   
This map is area preserving so the Haar measure $\mu$ is invariant. 
Moreover the function $g_{\alpha}(\zeta_{1})$ is of the form
$$
g_\alpha(\zeta_{1})=\frac{R(e^{2\pi i \alpha}\zeta_{1})}{R(\zeta_{1})}
$$
where $R$ is a \emph{measurable} function of modulus $1.$
Let $f(\zeta_1, \zeta_2):= R(\zeta_1)/\zeta_2.$ Then $f$ is invariant:

\bea
(f\circ T_\alpha)(\zeta_1,\zeta_2) & = & f(e^{2\pi i \alpha}\zeta_1, g_\alpha(\zeta_1)\zeta_2)\\
& = & \frac{R(e^{2\pi i \alpha}\zeta_1)}{g_\alpha(\zeta_1)\zeta_2}\\
& = & \frac{g_\alpha(\zeta_1)R(\zeta_1)}{g_\alpha(\zeta_1)\zeta_2}\\
& = & \frac{R(\zeta_1)}{\zeta_2}\\
& = & f(\zeta_1,\zeta_2)\\
\eea

Any such $f$ has essential range 
$$
E_{f}:=\{s_{0}\in S^1; 
\mu(f^{-1}(\{|s-s_{0}|<\delta\}))>0 \ 
{\mbox{for all}} \ \delta>0\}
$$
equal to all of $S^1$.  Let $s_{0}=\mathrm{e}^{2\pi i\theta_{0}}$ and let 
$I(s_{0},\delta)=
\{\mathrm{e}^{2\pi i\theta};|\theta-\theta_{0}|<\frac{\delta}{2}\}$.
If $R(\zeta_1)=e^{2\pi i \theta_1}$ is defined at a given point $\zeta_{1}$ the set of $\zeta_{2}=e^{2\pi i \theta_2}$ for which 
$f(\zeta_{1},\zeta_{2})\in I(s_{0},\delta)$ is given by 
$\{\mathrm{e}^{-2\pi i\theta}R(\zeta_{1});|\theta_1-\theta_0-\theta|<\frac{\delta}{2}\}$.
Letting $\chi=\chi_{s_{0},\delta}$ be the characteristic function for $K=K(s_{0},\delta):=f^{-1}(I(s_{0},\delta))$ we get that 
$$
\int_{K}\chi\mathrm{d}\mu=\int_{S^1}
(\int_{S^1}\chi\mathrm{d}\mu_{2})\mathrm{d}\mu_{1}
=\int_{S^1}\delta\mathrm{d}\mu_{1}=\delta.
$$
Hence any pair $(s_{0},\delta)$ gives rise to an invariant probability measure 
$\mu_{s_{0},\delta}:=\frac{\chi_{s_{0},\delta}\cdot\mu}{\delta}$. 
Let $\mu_{s_{0}}$ be a weak limit of measures $\mu_{s_{0},\delta_{j}}$
as $\delta_{j}\rightarrow 0$ and we get an invariant probability measure 
supported on the the level set $\{f=s_{0}\}$.  
It follows that there are uncountably many probability measures invariant under $T_\alpha.$ 
Observe that on $f^{-1}(s_0), T^n_\alpha(\zeta_1,\zeta_2)=(e^{2\pi i n\alpha}\zeta_1,s_0 R(e^{2\pi i n \alpha}\zeta_1))$ so the map is essentially an irrational rotation. Actually we have that 

\begin{cor}
The extremal probability measures invariant under $T_\alpha$ are parametrized 
by $S^1$; they  
are the pull backs of the 
Lebesgue
measure on $\{|\zeta_1|=1\}$ by $\pi_{1}$ to the graphs given by the level sets of $f.$ 
\end{cor}
\begin{proof}
By the limit process above there exists (at least) one invariant measure 
supported on each level set. On the other hand 
if $\nu$ is such an extremal measure then $(\pi_1)_*\nu$ is the Lebesgue measure on 
$\{|\zeta_1|=1\}.$  
Extremality then implies that $\nu$ has Borel support on a level set of $f.$ It is easy to show that $\{f=s_0\}$ which is a graph, supports a unique invariant measure.

\end{proof}

Having shown that any such map $T_{\alpha}$ gives rise to many invariant 
measures, hence also many closed directed currents, we proceed to construct
maps with dense orbits.  Note that 
$$
T_{\alpha}^n(\zeta_{1},\zeta_{2})=(\mathrm{e}^{2\pi in\alpha}\zeta_{1},
\Pi_{j=0}^{n-1}g_{\alpha}(\mathrm{e}^{2\pi ij\alpha}\zeta_{1})\zeta_{2})
= (\mathrm{e}^{2\pi in\alpha}\zeta_{1},\frac{R(\mathrm{e}^{2\pi in\alpha}\zeta_{1})}{R(\zeta_{1})}\zeta_{2}).
$$
In particular $R$ cannot be a continuous function if we want the orbits to be dense. 

\

Following Furstenberg we will construct the function $g$ as follows:

\bea
h(\theta) & := & \sum_{k\neq 0}\frac{1}{|k|} (e^{2\pi i n_k \alpha}-1)e^{2\pi i n_k \theta}\;{\mbox{and}}\\
g(e^{2\pi i \theta}) & := & e^{2 \pi i h(\theta)} \,
\eea
where the integers $n_{k}$ increase rapidly enough for the function $g$ (and $h$)
to be holomorphic on $\mathbb C^*$. We inductively define integers $v_k$ 
by 

\bea
v_1 & := & 1\\
v_{k+1}  & := & k\cdot 2^{v_k}+v_k+1\\
\eea

Let 
\bea
\alpha & := & \sum_{k=1}^\infty \frac{1}{2^{v_k}} \; {\mbox{and}}\\
n_k & := & 2^{v_k}, k>0\\
n_{-k} & = & -n_k, k>0.\\
\eea

It is easy to show that $\alpha$ is irrational.

\begin{remark}
Furstenberg defines $v_{k+1}=2^{v_{k}}+v_{k}+1$.  With this definition 
the map $h$ is only holomorphic on the annulus $\{\frac{1}{2}<|\zeta|<2\}$; 
the $k$ is added to get convergence on all of $\mathbb C^*$.  This is not needed
for all orbits to be dense. 
\end{remark}

Now
\bea
n_k\alpha-[n_k \alpha] & = & 2^{v_k} \sum_{\ell=k+1}^\infty \frac{1}{2^{v_{\ell}}}\\
& < & 2\cdot 2^{v_k} \frac{1}{2^{v_{k+1}}}\\
& \leq & 2 \cdot \frac{2^{v_k}}{2^{k\cdot 2^{v_k}}\cdot 2^{v_k} \cdot 2^1}\\
& = & \frac{1}{2^{k\cdot 2^{v_k}}}\\
& = & \frac{1}{ 2^{k\cdot n_k}}.\\
\eea

\begin{lemma}
The map $h$ and $g$ are holomorphic on $\mathbb C^*$.
\end{lemma}
\begin{proof}
Define the function 
$h_{+}(\zeta)
:=\underset{k>0}{\sum}\frac{1}{k}(\mathrm{e}^{2\pi in_{k}\alpha}-1)
\zeta^{n_{k}}$.
Then $|\mathrm{e}^{2\pi in_{k}\alpha}-1|\cdot|\zeta|^{n_{k}}
=|\mathrm{e}^{2\pi i (n_{k}\alpha - [n_{k}\alpha])}-1|\cdot |\zeta|^{n_{k}}
\leq 2\pi(\frac{|\zeta|}{2^k})^{n_{k}}$ and so by Abel's lemma $h_{+}$
is holomorphic on $\mathbb C$.  
Similarly the function $h_{-}(\zeta)
:=\underset{k<0}{\sum}\frac{1}{k}(\mathrm{e}^{2\pi in_{k}\alpha}-1)
\zeta^{n_{k}}$ is holomorphic on $\mathbb C^*$
We have that $g(\zeta) = \mathrm{e}^{2\pi i (h_{+}(\zeta)+h_{-}(\zeta))}$ 
for all $|\zeta|=1$.  
\end{proof}
\

To see that $g$ is on the desired form define $H\in L^2([0,1])$ by 

\bea
H(\theta) & = & \sum_{|k| \neq 0} \frac{1}{|k|} e^{2\pi i n_k \theta}. 
\eea
Then $h(\theta)= H(\theta+\alpha)-H(\theta)$, and so defining $R$ by 
\bea
R(e^{2\pi i \theta}) & = & e^{2\pi i  H(\theta)}
\eea
gives us $\frac{R(e^{2\pi i \alpha} \zeta)}{R(\zeta)}= g(\zeta)$

\begin{prop}\label{dense}
All orbits of the map $T_{\alpha}$ restricted to the torus are dense.  
\end{prop}
\begin{remark}
Furstenberg defines the function $H$ by $H(\Theta):=\underset{k\neq 0}{\sum}
\frac{1}{|k|}\mathrm{e}^{2\pi i n_{k}\lambda\Theta}$ for some unknown 
$\lambda$ in order to conclude that the function $R$ is merely measurable, not 
continuous.  As remarked before this is  already implied by the density of orbits.  
\end{remark}

\begin{proof}

We will prove this through a sequence of lemmas, and we start by 
calculating the iterates $T_{\alpha}^n$.

\begin{lemma}  
First 
$$
\sum_{j=0}^{n-1}h(\theta+j\alpha)= \sum_{k \neq 0} \frac{1}{|k|} e^{2\pi i n_k 
\theta}[ e^{2\pi i n_k n \alpha}-1]
$$
and so for $\zeta_1=\mathrm{e}^{2\pi i\theta}$ we have that 
$$
T_{\alpha}^{n}(\zeta_1,\zeta_2)=(e^{2\pi i n \alpha}
\zeta_1, e^{2\pi i \sum_{k \neq 0} \frac{1}{|k|} e^{2\pi i n_k 
\theta}[ e^{2\pi i n_k n \alpha}-1]}\zeta_2)
$$
\end{lemma}

\begin{proof}

\bea
\sum_{j=0}^{n-1}h(\theta+j\alpha) & = &
\sum_{j=0}^{n-1} \sum_{k\neq 0} 
\frac{1}{|k|} (e^{2\pi i n_k \alpha}-1)e^{2\pi i n_k(\theta+j\alpha)}\\
& = & \sum_{k \neq 0} \frac{1}{|k|} 
e^{2\pi i n_k \theta} \sum_{j=0}^{n-1}
(e^{2\pi i n_k \alpha}-1)e^{2\pi i n_k j \alpha}\\
& = & \sum_{k \neq 0} \frac{1}{|k|} 
e^{2\pi i n_k \theta}[ e^{2\pi i n_k n \alpha}-1]\\
\eea

\end{proof}

Set $m_s:=
\frac{n_{s+1}}{2^4 n_s}$.   

Then $m_{s}\cdot\alpha -[m_{s}\cdot\alpha]=\frac{2^{v_{s+1}}}{2^4 2^{v_{s}}}
\cdot\sum_{j=s+1}^\infty 2^{-v_{j}}<\frac{1}{2^{v_s+3}}<\frac{1}{2^s}$. 
Then for any point $(\zeta_{1},\zeta_{2})$ we have for 
$(\zeta_{1}^{s,l},\zeta_{2}^{s,l}):=T_{\alpha}^{m_{s+l}}\circ T_{\alpha}^{m_{s+l-1}}
\circ\cdot\cdot\cdot\circ T_{\alpha}^{m_{s}}(\zeta_{1},\zeta_{2})$ that 

\begin{itemize}
\item[$(*)$] $|\zeta_{1}-\zeta_{1}^{s,l}|<2\pi\sum_{j=s}^\infty 2^{-j}$.
\end{itemize}

Hence under these iterates the first coordinate does not change much.  Our goal 
now is to show that (under some assumption on $\zeta_{1}$ and the size of $s$) 
for the same iterates  
the second coordinate rotates around the circle with very small intervals, \emph{i.e.}, 
we can get very near to any point.  \

Let $\theta=\frac{r}{n_s}$ for some $r\in\mathbb R$.

\begin{lemma}\label{int}
There are fixed constants $0<a<b$, $c>0$ and $N\in\mathbb N$ such that if 
$|r|<c$ mod $\mathbb Z$ and $s\geq N$ then
$$
-\frac{b}{s}\leq \sum_{k \neq 0} 
\frac{1}{|k|}e^{2\pi i \frac{n_k r}{n_s}}(e^{2\pi i n_k 
\frac{n_{s+1}}{2^4 n_s}\alpha }-1)\le -\frac{a}{s}.
$$
Hence, for $\zeta_{1}=\mathrm{e}^{2\pi i\Theta}$, 
$$
T_{\alpha}^{m_{s}}(\zeta_1,\zeta_2)=(\zeta_1',\zeta_2')=
\left(e^{2\pi i m_{s}
\sum_{k=s+1}^\infty \frac{1}{n_k}}\zeta_1,e^{2\pi i u}\zeta_2
\right),
-\frac{b}{s}\leq u \leq -\frac{a}{s}.
$$
\end{lemma}

\begin{proof}
Write $T^{m_{s}}_{\alpha}(\zeta_{1},\zeta_{2})=(A_{s},B_{s})$.  By 
the previous lemma we have that 
$$
B_{s}=
\zeta_2{\mbox{exp}}\left[2\pi i\sum_{k \neq 0}\frac{1}{|k|}
e^{2\pi i n_k \frac{r}{n_s}}\left(e^{2\pi i n_k \frac{n_{s+1}}{2^4 n_s} 
\alpha}-1\right)\right]
$$
To estimate the size of $u$ as defined above we consider different 
parts of the sum inside the bracket separately.  For $s>|k|$ we will use that 

\bea
\left|\frac{1}{|k|}
e^{2\pi i n_k \frac{r}{n_s}}\left(e^{2\pi i n_k \frac{n_{s+1}}{2^4 n_s}\alpha}
-1\right)\right| &\leq &
\frac{1}{|k|} 2\pi \frac{n_k n_{s+1}}{2^4 n_s}\sum_{t=s+1}^\infty \frac{1}{n_t}\\
& \leq & \frac{n_k}{n_s},\\
\eea
and for $s<|k|$ we will use that 
\bea
\left|\frac{1}{|k|}
e^{2\pi i n_k \frac{r}{n_s}}\left(e^{2\pi i n_k \frac{n_{s+1}}{2^4 n_s} \alpha}
-1\right)\right|
& \leq & \frac{1}{|k|}2\pi \frac{n_{s+1}}{2^4n_s}n_{k}\sum_{t=k+1}^\infty \frac{1}{n_t}\\
& \leq & \frac{1}{|k|}2\pi \frac{n_{s+1}}{2^4n_s}\frac{n_{k}}{n_{k+1}}
\sum_{t=k+1}^\infty \frac{n_{k+1}}{n_{t}}\\
& \leq & \frac{n_{s+1}\cdot n_{k}}{n_s\cdot n_{k+1}}
\eea
Write 
$$
\sigma_{s}= \sum_{k \neq 0,|k|\neq |s|} 
\frac{1}{|k|}e^{2\pi i \frac{n_k r}{n_s}}(e^{2\pi i n_k 
\frac{n_{s+1}}{2^4 n_s}\alpha }-1)
$$
Then $|\sigma_{s}|<s\cdot 2^{1-s}$, and so there exists an $N\in\mathbb N$
such that if $s\geq N$ then 
$|\sigma_{s}|<\frac{1}{2s}(1-\mathrm{cos}(\frac{\pi}{8}))$ whenever $\delta$
is small enough:
We have that $\sum_{k=1}^{s-1}\frac{n_{k}}{n_{s}}\leq (s-1)\frac{n_{s-1}}{n_{s}}<(s-1)2^{-s}$, and   
furthermore that $\sum_{k=s+1}^\infty\frac{n_{s+1}\cdot n_{k}}{n_s\cdot n_{k+1}}
\leq\frac{1}{2^s}\sum_{k=s+1}^\infty\frac{1}{2^k}<\frac{1}{2^s}$.

To conclude the proof of the lemma we need an estimate for
$|k|=|s|$: 

\bea
& & \frac{1}{s}
e^{2\pi i n_s \frac{r}{n_s}}\left(e^{2\pi i n_s \frac{n_{s+1}}{2^4 n_s} \alpha}-1\right)\\
& + & \frac{1}{s}
e^{2\pi i (-n_s) \frac{r}{n_s}}\left(e^{2\pi i (-n_s) \frac{n_{s+1}}{2^4 n_s} \alpha}-1\right)\\
& = & \frac{1}{s}\left[e^{2\pi i r} \left (e^{2\pi i \frac{n_{s+1}}{2^4 } \alpha}-1\right)
 +  
e^{2\pi i (-r)}(e^{2\pi i  \frac{-n_{s+1}}{2^4 } \alpha}-1)\right]\\
& = &  \frac{1}{s}\left[ \left (e^{2\pi i \frac{n_{s+1}}{2^4 } \alpha}-1\right)
 +  
(e^{2\pi i  \frac{-n_{s+1}}{2^4 } \alpha}-1)\right]\\
& + &  \frac{1}{s}\left[(e^{2\pi i r}-1) \left (e^{2\pi i \frac{n_{s+1}}{2^4 } \alpha}-1\right)
 +  
(e^{2\pi i (-r)}-1)(e^{2\pi i  \frac{-n_{s+1}}{2^4 } \alpha}-1)\right]\\
& = & \frac{2}{s} [\cos(2\pi \frac{n_{s+1}}{2^4}\alpha)-1]+ E\\
& = & \frac{2}{s}[\cos \left( \frac{2\pi}{2^4}\left(1+\sum_{t=s+2}^\infty 
\frac{n_{s+1}}{n_t}\right)\right)-1] + E\\
|E| & = & \left|\frac{1}{s}\left[(e^{2\pi i r}-1) ((e^{2\pi i \frac{n_{s+1}}{2^4 } \alpha})-1)
 +  
(e^{2\pi i (-r)}-1)(e^{2\pi i  \frac{-n_{s+1}}{2^4 } \alpha}-1)\right]\right|\\
& \leq & \frac{1}{s} [2|e^{2\pi ir}-1|+2|e^{-2\pi i r} -1|]\\
\eea

If $|r|$ is small enough mod $\mathbb Z$ we see that $|E|<\frac{1}{2s}(1-\mathrm{cos}(\frac{\pi}{8}))$
and by possibly having to increase $N$ we may put $a\approx(1-\mathrm{cos}(\frac{\pi}{8}))$
and $b\approx3(1-\mathrm{cos}(\frac{\pi}{8}))$
\end{proof}

\begin{lemma}
There exists a constant $C$ such that the following holds.  
Let $r\in\mathbb R$ be some number, write $\Theta=\frac{r+k}{n_{s}}, k\in\mathbb Z$
and $\zeta_{1}^0=\mathrm{e}^{2\pi i\Theta}$.  Let $\zeta_{2}^0\in S^1$ 
be arbitrary.  Define $\zeta_{1}^j$ inductively by $(\zeta_{1}^j,\zeta_{2}^j):=
T^{m_{s+j-1}}(\zeta_{1}^{j-1},\zeta_{2}^{j-1})$.  Write $\zeta_{1}^j=\mathrm{e}^{2\pi i\Theta_{j}}$
and $\Theta_{j}=\frac{r^{(j)}}{n_{s+j}}$.  Then 
$$
r^{(j)}\underset{\mathrm{\mod}\mathbb Z}{=}r\frac{n_{s+j}}{n_{s}}+\nu_{j}
$$
where $\nu_{j}<C2^{-s}$.  In particular, if $s$ is big enough and $r=0$,
then $r^{(j)}<c$ for all $j$ (c is the constant from the previous lemma).
\end{lemma}
\begin{proof}
We claim first that 
$$
(*) \  r^{(j)}=r\frac{n_{s+j}}{n_{s}}+\sum_{l=1}^j [\frac{n_{s+l}}{2^4 n_{s+l-1}}
\sum_{t=s+l}^\infty\frac{n_{s+j}}{n_{t}}].
$$
We show this by induction, so assume that it holds for some $j$.  Then 
$$
\Theta_{j+1}=\Theta_{j} + \frac{n_{s+j+1}}{2^4 n_{s+j}}
\sum_{t=s+j+1}^\infty n_{t}^{-1} = \frac{r^{(j)}}{n_{s+j}}+
\frac{n_{s+j+1}}{2^4 n_{s+j}}
\sum_{t=s+j+1}^\infty n_{t}^{-1}, 
$$
and so 
\bea
r^{(j+1)} & = & \Theta_{j+1}\cdot n_{s+j+1}\\
& = & 
(r\frac{n_{s+j+1}}{n_{s}}+\sum_{l=1}^j [\frac{n_{s+l}}{2^4 n_{s+l-1}}
\sum_{t=s+l}^\infty\frac{n_{s+j+1}}{n_{t}}])+ 
\frac{n_{s+j+1}}{2^4 n_{s+j}}\sum_{t=s+l+1}^\infty\frac{n_{s+j+1}}{n_{t}}.
\eea
Now write $r^{(j)}=r\frac{n_{s+j}}{n_{s}}+\nu_{j}$ according to $(*)$.  Then 
mod$\mathbb Z$ we have that 
$$
\nu_{j}< j\cdot\frac{n_{s+j}^2}{2^3\cdot n_{s+j-1}\cdot n_{s+j+1}}<j\cdot 2^{-j}\cdot 2^{-s}.
$$
\end{proof}

\medskip

To finish the proof of Proposition 11, let $p\in \mathbb T^2$ and let $\mathcal O$ denote the closure of the orbit of $p$. We want to show that an arbitrary point $\zeta'=(\zeta_1',\zeta_2')=
(e^{2\pi i \Theta_1'},e^{2\pi i \Theta_2'})$ lines in $\mathcal O,$ i.e. for any $\epsilon>0$ we need to find a point in $\mathcal O$ closer to $\zeta'$ than $\epsilon.$ 

Let $s \geq N$ be big enough such that $\frac{b}{s}<\frac{\epsilon}{4\pi}$ and such that
$\sum_{j=s}^\infty 2^{-s}<\frac{\epsilon}{8\pi}.$ Note that by the irrationality of $\alpha$, the closure of the orbit of $p$ must contain a point $(\hat{\zeta}_1,\hat{\zeta}_2)$ for any $|\hat{\zeta}_1|=1.$ 
Let $(\zeta_1,\zeta_2)$ denote any point in $\mathcal O$ with 
$\zeta_1=e^{2\pi i \theta}, \theta=n/n_s$
with $n\in \mathbb Z$ such that $|\zeta_1-\zeta_1'|<\epsilon/4.$

Next we define 

$$
(\tilde{\zeta}_1,\tilde{\zeta}_2)=(e^{2\pi i \tilde{\Theta}_1},e^{2\pi i \tilde{\Theta}_2})
= T^{m_s}(\zeta_1,\zeta_2)$$

\noindent and  let

$$
(\tilde{\zeta}_1^\ell,\tilde{\zeta}_2^\ell)= T^{m_{s+\ell}}\circ \cdots \circ T^{m_{s+1}}(\tilde{\zeta}_1,\tilde{\zeta}_2).
$$

Then

\bea
|\tilde{\zeta}_1^\ell-\zeta_1'| & \leq & |\tilde{\zeta}_1^\ell-\zeta_1|+|\zeta_1-\zeta_1'|\\
& < & 2\pi \sum_{j=s}^\infty 2^{-j}+\frac{\epsilon}{4} \leq \frac{\epsilon}{2}.\\
\eea

Moreover, if $\tilde{\zeta}_2^{\ell}= e^{2\pi i \theta_\ell},$ then $\theta_\ell=\tilde{\theta}_2+\sum_{j=1}^\ell u_j$ with $-\frac{b}{s+j-1}<u_j<-\frac{a}{s+j-1}.$ It follows that for some $\ell$,
$|\tilde{\zeta}_2^\ell-\zeta_2'|<\frac{\epsilon}{2}.$\\

\end{proof}

Remark on the Furstenberg map: For a given sequence $\{n_k\}$ we can construct a family 
$T_\alpha$ of Furstenberg maps when $\alpha$ is in a dense $G_\delta$ set on the circle and all maps commute. In particular they have the same invariant function $f.$ If a homeomorphism commutes with some $T_\alpha$, then it is one of these.

\begin{thm}
There is a minimal lamination $\mathcal F$ by 
Riemann surfaces which admit uncountably many extremal closed currents which are 
mutually singular. Moreover every positive $\partial \overline{\partial}$-closed 
current is closed. The lamination can be extended holomorphically to a holomorphic bundle 
with $\mathbb C^*\times \mathbb P^1$ fiber and a given compact Riemann surface of genus $g \geq 2$ as a base. The minimal set is a compact CR manifold.
\end{thm}

\begin{proof}
Take a Riemann surface $S$ with genus $g \geq 2.$ Let $(\gamma_i)$ be a basis for 
$\pi_1(S).$  We can assume that $\gamma_1$ does not disconnect $S.$ 
Choose $F_1$ as a Furstenberg map and $F_i={\mbox{Id}}$ for $i \geq 2.$ This gives a representation of $\pi_1(S).$ Then we get a minimal lamination with fiber $T^2$ where the bundle is defined using only one transition function, the Furstenberg map extended biholomorphically to
$\mathbb C^* \times \mathbb P^1,$ \emph{i.e.}, one transition function on an open set surrounding $\gamma_1.$ 
The fact that a positive directed $\partial\overline\partial$-closed current is closed will be proved in Corollary 3 in the next paragraph.  
\end{proof}

\begin{remark}
In \cite{L2009} Lozano-Rojo has showed that a lamination constructed by 
Ghys-Kenyon \cite{Gh1999} admits two mutually singular transverse invariant measures. 
\end{remark}

\begin{remark}
To make a Riemann surface lamination over a torus by suspension,
define a biholomorphism of 
$\mathbb P^1\times [\mathbb C^* \times \mathbb P^1]$ given by
$(z,\zeta_1,\zeta_2) \overset{\Phi}{\rightarrow} (2z,T_{\alpha}(\zeta_1,\zeta_2))$
(we actually restrict the map to a biholomorphism of  $\mathbb C^*\times
[\mathbb C^* \times \mathbb P^1]$).
\end{remark}

\begin{remark}
When $g=1,$ every $\partial \overline{\partial}$-closed current is closed. Indeed $\tilde S$ is $\mathbb C$ and it has no nonconstant positive harmonic function.
\end{remark}

We now construct laminations with infinitely many mutually singular, non closed, 
$\partial \overline{\partial}$-closed currents. 

We consider the lamiation $\mathcal F$ constructed in Theorem 6 with a torus as base. 
It sits in a complex manifold and is a CR-manifold in a $(\mathbb C^* \times \mathbb P^1)$
fibration over a torus. 
It is a suspension defined using the biholomorphism of 
$\mathbb C^* \times [\mathbb C^* \times \mathbb P^1]$ 
given by $(z,\zeta_1,\zeta_2) \rightarrow^\phi (2z, T_\alpha(\zeta_1,\zeta_2)).$ 
and as we have seen possesses a family $\{T_c\}$ of mutually singular closed positive currents of mass $1.$ Now we construct a second foliation $\mathcal F',$ using the Furstenberg diffeomorphism and which has no closed positive directed currents.

Indeed we modify the above suspension, over a surface of genus $g \geq 2,$ using as a 
map $F_{\gamma_1}$, the Furstenberg map and as a map $F_{\gamma_2}$ a 
diffeomorphism of the torus $\mathbb T^2$, with an attractive fixed point. 
This destroys the directed positive closed currents, but there is a 
$\partial \overline{\partial}$-closed positive directed current, $T'$ of mass $1.$ 
A priori we do not know if it is unique. 
Observe however that the leaves of $\mathcal F'$ are minimal and that $\mathcal F'$ 
is a lamination in a complex manifold $Y'.$

\begin{thm}
The lamination $\mathcal F\otimes\mathcal F'$ in $Y \times Y'$ has infinitely many 
non-closed mutually singular $\partial \overline{\partial}$-closed positive (2,2)-currents.
\end{thm}

\begin{proof}
The leaves of $\mathcal F\otimes\mathcal F'$ are two dimensional complex surfaces and
 they are dense.  With the notations introduced above every current $T_c \otimes T'$ is 
 directed by $\mathcal F\otimes\mathcal F'$ and is 
 $\partial \overline{\partial}$-closed. Clearly they are mutually singular. 
\end{proof}

\section{$\partial \overline{\partial}$-closed directed currents}

We will give an abstract criterion for the existence of a $\partial \overline{\partial}$-closed current in the style of Sullivan \cite{S1976}.  See also \cite{FS2008, G1983, Gh1999}.

\begin{thm}
Let $(X,\mathcal L)$ be a compact, non-singular laminated set.   
Assume that the leaves are holomorphic of complex dimension $\ell \geq 1.$ 
Either there is a nonzero positive $\partial \overline{\partial}$-closed current 
directed by $\mathcal L$ or there is a $\partial \overline{\partial}$-exact $\ell$ 
volume form on leaves.
\end{thm}

\begin{proof}
Let $\mathcal C$ be the convex compact of positive currents of mass one directed by the 
lamination. Let 
$F=\{ i \partial \overline{\partial} \psi; \psi\in\mathcal A^{l-1}(\mathcal L)\}$, \emph{i.e.}, $\psi$
is continuous and smooth along leaves.   
If $\mathcal C$ 
intersects $F^\perp$ we have a non-zero $\partial\overline\partial$-closed current. 

Otherwise there exists by Hahn-Banach an element $\varphi$ of ${}^\perp(F^\perp)=\overline F$
such that $\varphi$ is strictly positive on $\mathcal C$.  This means that $\varphi$ is strictly positive 
along leaves.  For a sequence $i\partial\overline\partial\psi_{j}$ converging 
to $\varphi$ we get by compactness that $i\partial\overline\partial\psi_{j}$
is a volume form for j big enough.
\end{proof}

Note that it follows that a compact, non-singular Riemann surface lamination always 
carry a positive $\partial\overline\partial$-closed (1,1)-current - a result 
due to L. Garnett \cite{G1983}.  Due to the maximum principle such a lamination cannot carry
a $\partial\overline\partial$-exact volume form.  

\medskip

We will now give an example to show that in general, 
in contrast to the one dimensional case, a lamination in a K\"{a}hler manifold 
does not have 
to carry any positive $\partial\overline\partial$-closed current.  

\begin{prop}
Let $(\mathbb P^2,\mathcal L)$ be a lamination without algebraic leaves and with 
only hyperbolic singularities.  Then the product $(\mathbb P^2,\mathcal L)\times (\mathbb P^2,\mathcal L)$ does not carry any positive $\partial\overline\partial$-closed current.  
\end{prop}
\begin{proof}
It is shown in \cite{FS2008} that  $(\mathbb P^2,\mathcal L)$  admits a unique positive 
closed current $T$ of mass one and moreover that this current is not closed.

Let $\mathcal F_1\times\mathcal F_2$ denote the product $(\mathbb P^2,\mathcal L)\times (\mathbb P^2,\mathcal L)$ and 
assume to get a contradtiction 
that $T$ is a directed positive $\partial\overline\partial$-closed
current for this foliation.  

In a flow box $\mathbb B_{1}\times\mathbb B_{2}$ the current $T$ has the expression 
$$
T=\int h^{\alpha_{1},\alpha_{2}} [V^{\alpha_{1}}\times V^{\alpha_{2}}]\mathrm{d}\sigma(\alpha)
$$
where $\sigma$ is a measure on the transversal and $[V^{\alpha_{1}}\times V^{\alpha_{2}}]$
are the currents of integration on the corresponding plaques.  We prefer here to deal with the collection 
of measures $\mu:=h^{\alpha_{1},\alpha_{2}} (\mbox{volume measure on} \  V^{\alpha_{1}}\times V^{\alpha_{2}})
\mathrm{d}\sigma(\alpha)$.  We denote by $z_{1}$ the variable on the first factor $\mathbb P^2$
and by $z_{2}$ the variable on the second factor.  For each fixed $z_{1}$ we get a collection of 
measures $\mu_{z_{1}}(z_{2})$ by disintegration of $\mu$ on the corresponding flow box.  
The 
family $\mu_{z_{1}}$ define a harmonic current directed by $\mathcal F_2$
out of the singular points.  Lemma 17 below with Theorem 1.3 in \cite{DS2007}
shows that the current extends to a harmonic current through the discrete set sing($\mathcal F_2$).
By uniqueness for the normalized 
harmonic measure for $\mathcal F_2$ we get that this family of measures is independent of $z_{1}$.  
On the flow box disintegrate $\mu$ with respect to the first projection and let $\nu:=(\pi_{1})_{*}\mu$.
For a test function $\varphi(z_{1},z_{2})$ we have that 
$$
\int\varphi\mathrm{d}\mu=\int\mathrm{d}\nu(z_{1})\int\varphi(z_{1},z_{2})\mathrm{d}\mu_{z_{1}}(z_{2})
=\int\mathrm{d}\nu(z_{1})\int\varphi(z_{1},z_{2})\mathrm{d}\mu_{1}(z_{2}).
$$
So $\mu=\nu\otimes\mu_{1}$ on the space generated by $\varphi(z_{1})\psi(z_{2})$, hence 
$\mu=\nu\otimes\mu_1$.  Observe that $\nu$ is also a harmonic measure, \emph{i.e.}, orthogonal 
to function $\triangle_{\mathcal F_{1}}\varphi$.  Here $\triangle_{\mathcal F_{1}}$ 
denotes the laplacian along leaves.  Hence in a flow box we get that 
$h^{\alpha_{1},\alpha_{2}}(z_{1},z_{2})=h^{\alpha_{1}}(z_{1})\cdot k^{\alpha_{1}}(z_{2})$
with $h^{\alpha_{1}}, k^{\alpha_{1}}$ harmonic on leaves.   A pluriharmonic function cannot 
be the product of two harmonic functions except if one of them is constant.  This is a contradiction.  
\end{proof}

Bonatti and Gomez-Mont provided similar uniqueness results in $\mathbb P^1\times\mathbb P^1$.
They showed that when a Ricatti equation does not have a transverse invariant measure then the 
harmonic measure is unique on the limit set.  So we get a laminated set in $(\mathbb P^1)^4$
without a positive directed $\partial\overline\partial$-closed current of mass one.

\begin{lemma}
Let $X$ be a complex manifold and let $q\in X$ be a point.  Then any 
positive $\partial\overline\partial$-closed (1,1)-current $T$ on $X\setminus\{q\}$
has locally finite mass near $q$.  
\end{lemma}
\begin{proof}
We work in local coordinates so we assume that $T$ is given on $\mathbb B^k\setminus\{0\}$
in $\mathbb C^k$.  Let $\omega$ be a strictly positive test form.  We may write 
$\omega=i\partial\overline\partial\varphi$ and after a change of coordinates 
we may assume that $\varphi(z)=\sum_{j=1}^k\lambda_{j}|z_{j}|^2+o(\|z\|^2)$.  
For $\epsilon>0$ let $\varphi_{\epsilon}(z):=\mathrm{max}\{\varphi(z),\epsilon\}-\epsilon$
(or rather a smoothing of this function).  Then $i\partial\overline\partial\varphi_{\epsilon}$
is positive for $\epsilon$ small enough and $i\partial\overline\partial\varphi_{\epsilon}\rightarrow\omega$
uniformly on compact subsets of $\mathbb B^k\setminus\{0\}$ as $\epsilon\rightarrow 0$.
Let $\sigma$ be a smooth function such that $\sigma\equiv 1$ near $\frac{1}{2}\mathbb B^k$
and with compact support in $\mathbb B^k$.   Then $T(i\partial\overline\partial(\sigma\cdot\varphi_{\epsilon}))=0$
and so 
$$
T(\sigma\cdot i\partial\overline\partial(\varphi_{\epsilon}))\leq
|T(i\cdot\partial\sigma\wedge\overline\partial\varphi_{\epsilon})|+
|T(i\cdot\overline\partial\sigma\wedge\partial\varphi_{\epsilon})|+
|\varphi_{\epsilon}\cdot i\partial\overline\partial\sigma|.
$$
The right side of this inequality is independent of $\epsilon$ for small enough $\epsilon$. 
\end{proof}

\section{The functional $T\rightarrow i\tau \wedge \overline{\tau}\wedge T$}

Let $\mathcal{H(L)}$ denote the compact convex set of positive $\partial\overline\partial$-closed currents directed by $\mathcal L$ and of mass 1.

Suppose $T \in \mathcal H (\mathcal L)$.  Then $\overline \partial T= \frac{\overline{\partial}_b h_\alpha}
{h_\alpha} \wedge T,$ where $\overline{\partial}_b$ defines the $\overline{\partial}$ operator along leaves.  So $\overline{\partial} T= \overline{\tau} \wedge T$ with $\overline{\tau}$ well defined on $(0,1)$ forms along leaves.  By Harnack's inequality $\frac{\overline{\partial}_b h_\alpha}
{h_\alpha}$ is locally bounded.  Hence the form $i \tau \wedge \overline{\tau}$ is uniformly bounded, so $I(T)= \int i \tau \wedge \overline{\tau}\wedge T$ is finite. 

\begin{prop}
The functional $I(T)$ is continuous on $\mathcal H (\mathcal L)$ and is affine. The convex set  $\mathcal H_m(\mathcal L)$ where $I$ reaches its maximum is a face of $\mathcal H (\mathcal L).$ 
\end{prop}

\begin{proof}
The continuity is proved in \cite{FS2006}, Theorem 10. Suppose that $T_1,T_2$ are extremal points in $\mathcal H(\mathcal L)$, then

$$
I(aT_1+(1-a)T_2) = a I(T_1)+(1-a) I(T_2), 0 \leq a \leq 1.
$$

Indeed in a flow box the transverse measures are mutually singular. The forms $\tau_1,\tau_2$ are defined on disjoint saturated sets, hence

\bea
\overline{\partial} (a T_1+(1-a)T_2) & = & a\overline{\tau}_1 \wedge T_1+(1-a) \overline{\tau}_2 \wedge T_2\\
& = & (\overline{\tau}_1+\overline{\tau_2})\wedge(aT_1+(1-a)T_2)\\
\eea

It follows that

$$
I(aT_1+(1-a)T_2)=a I(T_1)+(1-a)I(T_2).
$$

We get a similar result for any convex combination of extremal elements. The fact that $I$ is affine follows by continuity. Define

$$
\mathcal H_m(\mathcal L) := \{T; I(T)=\max_{\mathcal H_m(\mathcal L)} I(T)\}.
$$

Clearly $\mathcal H_m (\mathcal L)$ is convex. If the segment $]T_1,T_2[$ touches $\mathcal H_m(\mathcal L)$ then it is contained in $\mathcal H_m(\mathcal L).$

\end{proof}

We want to apply the above result to prove that some suspension has no $\partial\overline\partial$-closed and non closed directed positive current.  We use the notation introduced in Section 3.

Let $S$ be a compact Riemann surface with universal covering $\tilde{S}.$ Consider the suspension
$X=\tilde{S} \times M/{\tilde\Gamma}.$ Assume the group $G=F(\pi_1(S))$ is commutative. Then this provides a commutative group $\tilde G$ of elements in Aut$(X)$, \emph{i.e.}, continuous automorphisms holomorphic along leaves.  
If $\pi$ denotes the canonical map: $\tilde{S} \times M \rightarrow X$, define
for $\gamma_0\in\pi_1(S)$, $H_{\gamma_0}(\pi(\tilde s,z)):=\pi(\tilde s,h_{\gamma_0}(z))$
with $h_{\gamma_0}=F(\gamma_0)$.  If $\pi(\tilde s_1,z_1)=\pi(\tilde s_2,z_2)$
we have that $\tilde s_2=\gamma(s_1)$ and $z_2=h_\gamma(z_1)$, and so 
$H_{\gamma_0}(\tilde s_2,z_2)=
H_{\gamma_0}(\gamma(s_1),h_\gamma(z_1))
=\pi(\tilde s_2,h_{\gamma_0}(h_\gamma(z_1)))
=\pi(\tilde s_2,h_\gamma(h_{\gamma_0}(z_1)))
=\pi(s_1,h_{\gamma_0}(z_1))$, since $G$ is commutative 
Hence the map is well defined, and it is holomorphic on leaves.

\begin{thm}
Suppose $X=\tilde{S}\times M/{}_{\tilde{\Gamma}}$ is a suspension as above, \emph{i.e.}, 
with a commutative group of self maps $G.$ If $T$ is a positive directed 
$\partial \overline{\partial}$-closed current, than $T$ is closed.
\end{thm}

\begin{proof}
Consider ${\mathcal H}_m(\mathcal L) $, the convex compact set of $\partial \overline{\partial}$-closed positive currents where the functional $T\rightarrow I(T)$ reaches its maximum m.  We want to prove that $m=0$, hence all the currents are closed.   The change of variable formula implies that $I(H_\gamma^*S)=I(S).$  Hence ${\mathcal H}_m(\mathcal L) $ is invariant under the maps $H_\gamma^*$.  The maps $H_\gamma^*$ acts on the currents belonging to $\mathcal H_m(\mathcal L)$ since $H_\gamma$ is holomorphic along leaves.  
The maps $H_\gamma^*$ act continuously on that convex compact set.
By Kakutani's theorem \cite{Z1984}  there is a fixed point $S$ in $\mathcal H_m(\mathcal L)$
invariant under $G,$ \emph{i.e.}, invariant under holonomy. It follows that $S$ is closed and hence $m=0.$

\end{proof}

\begin{cor}
The suspension obtained using irrational rotation, resp. the Furstenberg 
foliation, does not admit $\partial \overline{\partial}$-closed non closed positive directed currents.
\end{cor}

\bigskip

\noindent John Erik Forn\ae ss\\
Mathematics Department\\
The University of Michigan\\
East Hall, Ann Arbor, MI 48109\\
USA\\
fornaess@umich.edu\\

\noindent Nessim Sibony\\
Mathematics Department\\
Universit\'e Paris-Sud 11\\
Batiment 425\\
Orsay Cedex\\
France\\
nessim.sibony@math.u-psud.fr\\

\noindent Erlend Forn\ae ss Wold\\
Universitetet i Oslo\\
Matematisk Institutt\\
Postboks 1053\\
Blindern\\
No-0316, Oslo\\
Norway\\
erlendfw@math.uio.no\\

\end{document}